\documentclass{article}

\usepackage{fullpage}
\usepackage{Shortcuts_OPT}
\usepackage{bm, amsmath, amsthm}
\usepackage[T1]{fontenc}    
\usepackage[colorlinks=true, citecolor = blue]{hyperref}         
\usepackage{url}            
\usepackage{booktabs}       
\usepackage{amsfonts}       
\usepackage{nicefrac}       
\usepackage{microtype}      
\usepackage{xcolor}  
\usepackage{graphicx}
\usepackage{natbib}
\usepackage{enumitem}
\usepackage{graphicx}

\newtheorem{lemma}{Lemma}

\newtheorem*{theorem*}{Theorem}
\newtheorem{theorem}{Theorem}
\newtheorem{definition}[theorem]{Definition}
\newtheorem{corollary}{Corollary}

\newtheorem{remark}{Remark}
\newtheorem{example}{Example}

\newtheorem{Assumption}{\textbf{A}\!\!}

\newtheorem{manualtheoreminner}{A}
\newenvironment{manualtheoremA}[1]{%
  \renewcommand\themanualtheoreminner{#1}%
  \manualtheoreminner
}{\endmanualtheoreminner}

\newtheorem{manualtheoremcore}{Theorem}
\newenvironment{manualtheoremT}[1]{%
  \renewcommand\themanualtheoremcore{#1}%
  \begin{manualtheoremcore}
}{\end{manualtheoremcore}}

\newcommand*{\htwai}[1]{\textbf{\textcolor{red}{To: #1}}}

\allowdisplaybreaks[4]

\usepackage[normalem]{ulem}
\newcommand{\titlename}{Clipped SGD Algorithms for Performative Prediction: Tight Bounds for Clipping Bias and Remedies}
\title{\titlename}

\author{Qiang Li, Michal Yemini\thanks{M.~Yemini is with the faculty of Engineering, Bar-Ilan University, Israel. Email: \texttt{michal.yemini@biu.ac.il}}, Hoi-To Wai \thanks{Q.~Li and H.-T.~Wai are with the Department of Systems Engineering and Engineering Management, The Chinese University of Hong Kong, Hong Kong SAR of China. Emails: \texttt{\{liqiang, htwai\}@se.cuhk.edu.hk}
}}

\begin{document}
\maketitle

\begin{abstract}
This paper studies the convergence of clipped stochastic gradient descent (SGD) algorithms with decision-dependent data distribution. Our setting is motivated by privacy preserving optimization algorithms that interact with performative data where the prediction models can influence future outcomes. This challenging setting involves the non-smooth clipping operator and non-gradient dynamics due to distribution shifts. We make two contributions in pursuit for a performative stable solution using clipped SGD algorithms. First, we characterize the clipping bias with projected clipped SGD (PCSGD) algorithm which is caused by the clipping operator that prevents PCSGD from reaching a stable solution. When the loss function is strongly convex, we quantify the lower and upper bounds for this clipping bias and demonstrate a bias amplification phenomenon with the sensitivity of data distribution. When the loss function is non-convex, we bound the magnitude of stationarity bias. Second, we propose remedies to mitigate the bias either by utilizing an optimal step size design for PCSGD, or to apply the recent DiceSGD algorithm \citep{zhang2023differentially}. Our analysis is also extended to show that the latter algorithm is free from clipping bias in the performative setting. Numerical experiments verify our findings.
\end{abstract}

\section{Introduction}
A recent line of research in statistical learning is to analyze the behavior of stochastic gradient (SGD) type algorithms in tackling stochastic optimization problems with decision-dependent distributions. The latter can be motivated by the training of prediction models under distribution shifts \citep{quinonero2022dataset} where the data may `react' to the changing prediction models.
A common application scenario is that the training data involve human input that responds strategically to the model \citep{hardt2016strategic}. Distribution shifts affect the convergence of SGD-type algorithms and their efficacy as the distributions of gradient estimates vary gradually. 
The modeling of such behavior has led to the \emph{performative prediction problem} \citep{perdomo2020performative}; see the recent overview \citep{hardt2023performative}.

For the performative prediction problem, there is a growing literature on developing stochastic algorithms to find a fixed point solution of the repeated risk minimization procedure, also known as the performative stable (PS) solution(s). In the absence of knowledge of the distribution shift, the latter is the natural solution obtained by stochastic approximation algorithms. In particular, most prior works have only focused on settings when the loss function is strongly convex, e.g., \citep{mendler2020stochastic, drusvyatskiy2023stochastic, brown2022performative, li2022state}.
When the loss function is non-convex, the existing results are much rarer. To list a few, \citep{mofakhami2023performative} assumes a least-square like loss function; \citep{li2024stochastic} provided guarantees for general non-convex performative prediction.  
As an alternative, prior works also studied stochastic algorithms for finding performative optimal solutions by an extra step that estimates the form of distribution shifts \citep{izzo2021learn, miller2021outside, narang2023multiplayer}. 

However, most of the existing works have focused on analyzing SGD algorithms admitting a smooth drift term. 
An open problem in the performative prediction literature is to analyze the behavior of \emph{clipped SGD} algorithms which limit the magnitude of the stochastic gradient at every update step; {this in turn distorts the considered distribution shift compared with the nominal dynamics with no clipping}. 
{Gradient clipping is used to deal with multiple obstacles in learning algorithms such as}
the need for privacy preservation \citep{abadi2016deep}, dealing with gradient explosion in non-smooth learning \citep{shor2012minimization} such as training neural networks \citep{mikolov2012statistical, Zhang2020Why}, solving quasi-convex problems \citep{hazan2015beyond}, etc. Despite the difficulty with treating the non-smooth drift term with clipping, the clipped SGD algorithm has been analyzed in a multitude of works in the standard  i.i.d.~sampling setting with non-decision-dependent data. 
In fact, \cite{mai2021stability} studied the non-smooth optimization setting, and \cite{gorbunov2020stochastic} analyzed the convergence under heavy-tailed noise. 
On the other hand, it was found in \citep{chen2020understanding} that the clipped SGD may exhibit an asymptotic bias, i.e., deviation from the optimal/stationary solution, for asymmetric gradient distribution, and later on \cite{koloskova2023revisiting} showed that the bias can be unavoidable. This led to several recent works which considered \emph{debiasing} clipped SGD, e.g., \citep{khirirat2023clip21} studied bias-free clipped SGD for distributed optimization, \citep{zhang2023differentially} studied a differentially private SGD algorithm using two clipping operators simultaneously.

The current paper aims to study the convergence of clipped SGD algorithms in the \emph{performative prediction} setting. We focus on the interplay between the distribution shift and the non-smooth clipped updates. 
Our main contributions are:
\begin{itemize}[leftmargin=*,topsep=0mm, itemsep=.05mm]
    \item We show that {\pcsgd} converges in expectation to a neighborhood of the performative stable solution, a fixed point studied by \citet{perdomo2020performative}. For strongly convex losses, while the convergence rate is ${\cal O}(1/t)$ where $t$ is the iteration number, we show that the clipping operation induces an \emph{asymptotic clipping bias}. For non-convex losses, we show that in $T$ iterations, the scheme converges at the rate of ${\cal O}(1/\sqrt{T} )$ towards a biased stationary performative stable solution. 
    In both cases, we show that the magnitude of bias is proportional to the sensitivity of distribution shift and clipping threshold.
    \item For the case of strongly convex losses, we further show that there exists a matching lower bound for the asymptotic clipping bias upon specifying the class of performative risk optimization problem. Together with the derived upper bound, we demonstrate a \emph{bias amplification} effect of {\pcsgd} when subject to distribution shift in performative prediction.
    \item 
    As a remedy to the bias effect of {\pcsgd}, we study the recently proposed {\dicesgd} algorithm \citep{zhang2023differentially}. 
    We show that with a doubly clipping mechanism on both the gradient and clipping error, the algorithm can converge  exactly (nearly) to the PS solution under strongly convex (non-convex) loss.
\end{itemize}
Further, we show that there exists a tradeoff between the differential privacy guarantee and computation complexity that affects the optimal step size selection. 
The paper is organized as follows. \S\ref{sec:p} introduces the performative prediction problem and the {\pcsgd} algorithm. \S \ref{sec:main} presents the theoretical analysis for {\pcsgd}. \S\ref{sec:dice} discusses the {\dicesgd} algorithm and how it mitigates the clipping bias. {Our analysis includes both  strongly convex and non-convex settings}.
\S\ref{sec:experiment} presents the numerical experiments. \vspace{+.3cm}

{\bf Notations.}
Let $\RR^d$ be the $d$-dimensional Euclidean space equipped with inner product $\Pscal{\cdot}{\cdot}$ and induced norm $\norm{x} = \sqrt{\Pscal{x}{x}}$. 
$\EE[\cdot]$ denotes taking expectation w.r.t all randomness, $\EE_{t}[\cdot] \eqdef \EE_{t}[\cdot | {\cal F}_{t}]$ means taking conditional expectation on filtration ${\cal F}_t \eqdef \sigma(\{\prm_0, \prm_1, \cdots, \prm_t\})$, where $\sigma(\cdot)$ is the sigma-algebra generated by the random variables in the operand.
\vspace{-.1cm}

\section{Problem Setup}\label{sec:p}\vspace{-.1cm}
This section introduces the performative prediction problem and a simple projected clipped SGD algorithm. 
Our goal is to learn a prediction/classification model $\prm \in {\cal X}$ via the stochastic optimization  problem: 
\beq\label{eq:p}
\textstyle \min_{\prm\in {\cal X}} ~\EE_{Z\sim {\cal D}(\prm)} [\ell(\prm; Z)],
\eeq
where ${\cal X} \subseteq \RR^d$ is a closed convex set and $\ell(\cdot)$ is differentiable w.r.t.~$\prm$. An intriguing feature of \eqref{eq:p} is that the optimization problem is defined along with a \emph{decision-dependent distribution} ${\cal D}(\prm)$ where the distribution of the sample $Z \sim {\cal D}(\prm)$ depends on $\prm$. For example, it may take the form of the best response for a utility function parameterized by $\prm$. This setup models a scenario where the prediction model may influence the outcomes it aims to predict, also known as the \emph{performative prediction} problem; see \citep{perdomo2020performative, hardt2023performative}. 

The challenge in tackling \eqref{eq:p} lies in that the decision variable $\prm$ appears in both the loss function $\ell(\prm; Z)$ and the distribution ${\cal D}(\prm)$. As a result, \eqref{eq:p} is in general \emph{non-convex} even if $\ell(\cdot)$ is (strongly) convex. 
To this end, a remedy is to study the fixed point solutions deduced from tackling the partial optimization of minimizing $\EE_{Z\sim {\cal D}(\Bprm)} [\ell(\prm; Z)]$ w.r.t.~$\prm$ when the distribution depends on a fixed $\Bprm$. 

When $\ell(\cdot)$ is strongly convex in $\prm$, a popular solution concept is the performative stable (PS) solution \citep{perdomo2020performative}:
\begin{definition}\label{def:ps}
The solution $\prm_{PS} \in {\cal X}$ is called a PS solution to \eqref{eq:p} if it satisfies
    \beq\label{def:proj_thetaps} \textstyle
    \prm_{PS} = \argmin_{ \prm \in {\cal X}}~ \EE_{Z\sim {\cal D}(\prm_{PS})} [\ell(\prm; Z)].
    \eeq
\end{definition}
Note that $\prm_{PS}$ is unique and well-defined provided that (i) the loss function is smooth, (ii) ${\cal D}(\prm)$ is not overly sensitive to shifts in $\prm$; see \S\ref{sec:anal} for details. 

Alternatively, when $\ell(\cdot)$ is non-convex in $\prm$ and ${\cal X} \equiv \RR^d$, a recent solution concept is the stationary PS (SPS) solution \citep{li2024stochastic}:
\begin{definition}\label{def:sps}
The solution $\prm_{SPS} \in {\cal X}$ is called a stationary PS solution to \eqref{eq:p} if it satisfies
    \beq\label{def:proj_sps} \textstyle
    \EE_{Z\sim {\cal D}(\prm_{SPS})} [ \grd \ell(\prm_{SPS}; Z)] = {\bm 0}.
    \eeq
\end{definition}
Note that if $\ell(\cdot)$ is strongly convex, \eqref{def:proj_sps} will recover the definition of PS solution in \eqref{def:proj_thetaps}. 

It is clear that $\prm_{PS}$, $\prm_{SPS}$ do not solve \eqref{eq:p}, nor are they stationary solutions of \eqref{eq:p}. However, they remain reasonable estimates to the solutions of \eqref{eq:p}. As shown in \citep[Theorem 4.3]{perdomo2020performative}, the disparity between $\prm_{PS}$ and the optimal solution to \eqref{eq:p} is bounded by the sensitivity of the decision-dependent distribution (cf.~A\ref{assu:w1}).

To search for $\prm_{PS}$ or $\prm_{SPS}$, SGD-based schemes with greedy deployment, i.e., the learner deploys the latest model at the population after every SGD update, have 
been widely {deployed,}
see \citep{mendler2020stochastic, drusvyatskiy2023stochastic, li2024stochastic}. 

Motivated by privacy preserving optimization \citep{abadi2016deep} and improved stability in training deep neural networks \citep{Zhang2020Why}, in this paper we are interested in SGD algorithms which clip the stochastic gradient at each update. 
With the greedy deployment scheme in mind, our first objective is to study the projected clipped SGD ({\algoname}) algorithm for performative prediction:
\begin{align}\label{algo:pcsgd}
    \textbf{ \hspace{-.3cm} Deploy: } Z_{t+1} & \sim {\cal D}( \prm_t ),\\
    \textbf{\hspace{-.3cm} Update: }\prm_{t+1} & = {\cal P}_{\cal X} \big(\prm_t -  \gamma_{t+1} \clip_c( \grd \ell(\prm_t; Z_{t+1})) \big), 
\end{align}
where ${\cal P}_{\cal X}(\cdot)$ denotes the Euclidean projection operator onto ${\cal X}$, and $\clip_c(\cdot)$ is the clipping operator: for any ${\bm g} \in \RR^d$,
\beq \label{eq:clip_op}
   \clip_{c} ( {\bm g} ): {\bm g} \in \RR^d \mapsto \min\left\{1, {c} / {\norm{\bm g}_2} \right\} {\bm g},
\eeq
such that $c>0$ is a clipping parameter. Notice that if $c \to \infty$, \eqref{algo:pcsgd} is reduced to the projected SGD algorithm for performative prediction analyzed in \citep{mendler2020stochastic, drusvyatskiy2023stochastic}.

For $c < \infty$ and especially when $c < \| \grd \ell(\prm_t; Z_{t+1}) \|$, the {\algoname} recursion pertains to a non-gradient dynamics with non-smooth drifts due to the clipping operator; see Sec.~\ref{sec:main}. Prior analysis of {\algoname} are no longer applicable in this scenario. 




\begin{remark}
Under the performative prediction setting, \cite{drusvyatskiy2023stochastic} considered an alternative clipping model which approximates the loss function $\ell(\prm; z)$ by a linear model $\ell_{\prm'}(\prm; z) \eqdef \ell(\prm^\prime; z) + \pscal{\grd \ell(\prm^\prime; z)}{\prm - \prm^\prime}$, then updates $\prm$ by applying a proximal gradient method on $\ell_\prm(\prm; z)$ within a bounded set. This clips the models under training $\prm$ instead of clipping $\grd \ell(\prm; z)$ as in \eqref{algo:pcsgd}. Such algorithm belongs to the class of model based gradient methods whose fixed point is $\prm_{PS}$ when $\ell(\cdot)$ is strongly convex.
\end{remark}

\section{Main Results for {\pcsgd}}\label{sec:main}
In pursuit for a stochastic algorithm that finds $\prm_{PS}$ or $\prm_{SPS}$, we first study the convergence properties of {\pcsgd}. Additionally, we examine the tradeoff between model efficacy and privacy preservation of the algorithm.

\subsection{Analysis of the {\pcsgd} Algorithm} \label{sec:anal}
The analysis of \eqref{algo:pcsgd} involves challenges that are unique to the decision-dependent distributions. Curiously, the expectation of the \emph{unclipped} stochastic gradient $\grd \ell( \prm_t; Z_{t+1} )$ is not a gradient. To see this, consider the special case of $\ell( \prm ; Z ) = (1/2) \| \prm - Z \|^2$ and $Z \sim {\cal D}( \prm ) \Leftrightarrow Z \sim {\cal N}( {\bm A} \prm ; {\bm I} )$. Observe that $\EE_t [ \grd \ell( \prm_t; Z_{t+1} ) ] = ( {\bm I} - {\bm A} ) \prm_t$ has a Jacobian of ${\bm I} - {\bm A}$ which is asymmetric whenever ${\bm A}$ is asymmetric.  
Analyzing {\pcsgd} requires studying a non-gradient dynamics with non-smooth drifts induced by the clipping operator.
In the subsequent discussion, we analyze {\pcsgd} through identifying a suitable Lyapunov function depending on properties of the loss function $\ell(\cdot)$.

We define the shorthand notation:
\begin{equation}\label{eq:convention}
    \begin{aligned}
     f(\prm_1, \prm_2) &\eqdef \EE_{Z\sim {\cal D}(\prm_2)}[  \ell(\prm_1; Z)].
    \end{aligned}
\end{equation}
Unless otherwise specified, the vector $\grd f( \prm_1; \prm_2 )$ refers to the gradient taken w.r.t.~the first argument $\prm_1$.
We begin by stating a few assumptions pertaining to the performative prediction problem \eqref{eq:p}:
\begin{Assumption} \label{ass:scvx}
For any $\bar{\prm} \in {\cal X}$, the function $f(\prm; \bar{\prm})$ is {$\mu$} strongly convex w.r.t.~$\prm$, i.e., for any $\prm', \prm \in {\cal X}$,\vspace{-.1cm}
\beq \notag
f(\prm'; \bar{\prm}) \geq f( \prm ; \bar{\prm} ) + \pscal{ \grd f (\prm; \bar{\prm} ) }{ \prm' - \prm } + {\textstyle \frac{\mu}{2}} \norm{ \prm' - \prm }^2.
\eeq 
\end{Assumption}
\begin{Assumption} \label{assu:lips}
The gradient map $\grd \ell(\cdot; \cdot )$ is $L$-Lipschitz, i.e., for any $\prm_1,\prm_2 \in {\cal X}$, $z_1, z_2 \in {\sf Z}$,
\beq \notag 
\begin{aligned} 
& \| \grd \ell( \prm_1; z_1 ) - \grd \ell( \prm_2; z_2 ) \|  \leq L \big( \| \prm_1 - \prm_2 \| + \| z_1 - z_2 \| \big)
\end{aligned}
\eeq 
Moreover, there exists a constant $\ell^{\star}>-\infty$ such that $\ell(\boldsymbol{\theta} ; z) \geq \ell^{\star}$ for any $\boldsymbol{\theta} \in {\cal X}$.
\end{Assumption}
The above assumptions are common in the literature, e.g., \citep{perdomo2020performative, mendler2020stochastic, drusvyatskiy2023stochastic}. 
In addition, we require that
\begin{Assumption}\label{assu:bndgrd}
There exists a constant $G \geq 0$ such that
$\sup_{\prm \in {\cal X}, z \in {\sf Z}}\norm{\grd \ell(\prm; z)} \leq G$.
\end{Assumption}
This condition can be satisfied if ${\cal X}$ is compact; or for cases such as the sigmoid loss functions. 
Notice that a similar condition is used in \citep{Zhang2020Why}.
In some cases, we will use the following standard variance condition to obtain a tighter bound:
\begin{Assumption}\label{assu:var-ncvx}
There exists constants $\sigma_0, \sigma_1 \geq 0$ such that for any $\prm_1, \prm_2\in {\cal X}$, it holds
\begin{align*}
    & \EE_{Z \sim {\cal D}(\prm_2)} \left[ \norm{\grd \ell(\prm_1; Z) - \grd f(\prm_1; \prm_2)}^2 \right] \leq \sigma_0^2 + \sigma_1^2 \norm{\grd f(\prm_1; \prm_2)}^2.
\end{align*}
\end{Assumption}

\paragraph{Strongly Convex Loss} We first discuss the convergence of {\pcsgd} towards $\prm_{PS}$ under A\ref{ass:scvx}. We provide both \emph{upper} and \emph{lower} bounds for the asymptotic bias of {\pcsgd}, and demonstrate a \emph{bias amplification} effect as the sensitivity parameter of the distribution $\beta$ increases (cf.~A\ref{assu:w1}). Our result gives the first tight characterization of the bias phenomena in the literature.

To establish the convergence of {\algoname} in this case, we will need the following additional assumptions:
\begin{Assumption} \label{assu:w1} There exists $\beta \geq 0$ such that 
\[
W_1( {\cal D}( \prm) , {\cal D}( \prm' ) ) \leq \beta \| \prm - \prm' \|,~\forall~\prm, \prm' \in {\cal X}.
\]
Notice that $W_1(\cdot, \cdot)=\inf _{J \in \mathcal{J}(\cdot, \cdot)} \mathbb{E}_{\left(z, z^{\prime}\right) \sim J}\left[\left\|z-z^{\prime}\right\|_1\right]$ is the Wasserstein-1 distance,
where $\mathcal{J}\left(\mathcal{D}(\prm), \mathcal{D}\left(\prm^{\prime}\right)\right)$ is the set of all joint distributions on ${\sf Z} \times {\sf Z}$ whose marginal distributions are $\mathcal{D}(\prm), \mathcal{D}\left(\prm^{\prime}\right)$.
\end{Assumption}
We emphasize that $\beta$ in A\ref{assu:w1} quantifies the \emph{sensitivity} of the distribution against perturbation with respect to the decision model $\prm$. It will play an important role for the analysis below.
Notice that A\ref{ass:scvx}, \ref{assu:lips}, \ref{assu:w1} imply that $\| \prm^\star - \prm_{PS} \| \leq \frac{2 L \beta}{ \mu }$ where $\prm^\star \in \Argmin_{\prm\in {\cal X}} \EE_{Z\sim {\cal D}(\prm)} [\ell(\prm; Z)]$ is an optimal solution to the performative risk minimization problem \citep{perdomo2020performative}.

The first result upper bounds the {squared norm of the} error $\hat{\prm}_{t} \eqdef \prm_{t} - \prm_{PS}$ at the $t$th iteration of {\pcsgd} in expectation:
\begin{theorem}\label{thm1}(Upper bound) Under A\ref{ass:scvx}, \ref{assu:lips}, \ref{assu:bndgrd}, \ref{assu:w1}. Suppose that $\beta <\frac{\mu}{L}$, the step sizes $\{\gamma_{t}\}_{t\geq 1}$ are non-increasing and satisfy
i) $\textstyle \frac{\gamma_{t-1}}{\gamma_{t}} \leq 1 + \frac{\mu -L\beta}{2}\gamma_t, $ and ii)
$\gamma_{t} \leq \frac{2}{\mu - L\beta}$.
Then, for any $t\geq 1$, the expected squared distance between $\prm_t$ and the performative stable solution $\prm_{PS}$ satisfies
\begin{align}\label{eq:thm1}
    \hspace{-.2cm} \EE \normtxt{\hat{\prm}_{t+1}}^2  \leq  \prod_{i=1}^{t+1}(1- \tmu \gamma_{i}) \normtxt{\hat{\prm}_0}^2 + \frac{2 c_1}{\tmu} \gamma_{t+1} + \frac{8 {\cal C}_1}{\tmu^2},
\end{align}
where $c_1 \eqdef 2(c^2 + G^2) + d \sigmaDP{2}$, ${\cal C}_1 \eqdef (\max\{G-c, 0\})^2$, and $\tmu := \mu - L \beta$.
\end{theorem}
The proof is relegated to \S\ref{app:pf_scvx}. 
We remark that A\ref{assu:bndgrd} assumes that the stochastic gradient is uniformly bounded, where $G$ also accounts for the variance of stochastic gradient. If the 
additional assumption such as A\ref{assu:var-ncvx}
holds, it can be proven that \eqref{eq:thm1} holds with $c_1 = {\cal O}( c^2 + \sigma_0^2 )$. 
Our bound highlights the dependence on $t$ and the distribution shift parameter $\beta$.

From \eqref{eq:thm1}, as $t \to \infty$, with a properly tuned step size, the first term decays sub-exponentially to zero, the second term also converges to zero if $\gamma_t = {\cal O}(1/t)$. Regardless of the step size choice, the last term never vanish. It indicates an asymptotic clipping bias of {\pcsgd} and coincides with the observation in \cite{koloskova2023revisiting} for non-decision-dependent distribution. 

As $t \to \infty$, the bound in \eqref{eq:thm1} converges to a non-vanishing clipping bias term that reads
\beq \label{eq:bias_term_thm} 
{\tt Bias} = { 8 {\cal C}_1 } / { \tmu^2 } = {\cal O}( 1 / (\mu - L \beta)^2 ).
\eeq 
Even when $\beta = 0$, this bias term is improved over \citep{Zhang2020Why} which obtained an ${\cal O}(1/c)$ scaling. 
It turns out that the above characterization of the bias is \emph{tight} w.r.t.~$\mu-L\beta$. We {observe}:
\begin{theorem}\label{thm:lb}(Lower bound)
    For any clipping threshold $c \in (0, G)$, there exists a function $\ell(\prm; Z)$ and a decision-dependent distribution ${\cal D}(\prm)$ satisfying  A\ref{ass:scvx}, \ref{assu:lips}, \ref{assu:bndgrd}, \ref{assu:w1}, such that for all fixed-points of {\pcsgd} $\prm_{\infty}$
    satisfying $\EE_{Z\sim {\cal D}(\prminf)} [ \clip_{c} (\grd \ell(\prminf; Z)) ] = {\bm 0}$, it holds that 
    \beq 
        \norm{\prm_{\infty} - \prm_{PS}}^2 = \Omega \left( 1 / {(\mu - L\beta)^2} \right).
    \eeq 
\end{theorem}
The proof is relegated to \S\ref{app:lb}. 

Provided that $\beta < \frac{\mu}{L}$, Theorems \ref{thm1} and \ref{thm:lb} show that {\pcsgd} admits a clipping bias\footnote{The lower bound in Theorem~\ref{thm:lb} holds for any $\beta \geq 0$. When $\beta \geq \frac{\mu}{L}$, {\pcsgd} may not converge.} of $\Theta( 1 / (\mu - L\beta)^2 )$.
It illustrates a \emph{bias amplification} effect where as the \emph{sensitivity level} of distribution $\beta$ increases, the bias will increase as $\beta \uparrow \frac{\mu}{L}$. This is a unique phenomenon to performative prediction where in addition to the clipping level, the data distribution contributes to the bias.

\paragraph{Non-convex Loss} Next, we discuss the convergence of {\pcsgd} when the loss function is smooth but possibly non-convex, i.e., without A\ref{ass:scvx}. 

Our study concentrates on the case where ${\cal X} \equiv \RR^d$ such that the projection operator is equivalent to an identity operator. 
We consider the following condition:
\begin{Assumption} \label{assu:tv} There exists $\beta \geq 0$ such that 
\[
d_{TV}( {\cal D}( \prm) , {\cal D}( \prm' ) ) \leq \beta \| \prm - \prm' \|,~\forall~\prm, \prm' \in {\cal X}.
\]
where $d_{TV}( {\cal D}( \prm) , {\cal D}( \prm' ) )$ denotes the total variation (TV) distance between the distributions ${\cal D}( \prm) , {\cal D}( \prm' )$.
\end{Assumption}
The interpretation of $\beta$ is similar to that of A\ref{assu:w1}.
Notice that as $d_{\rm TV}( \mu , \upsilon ) \geq W_1( \mu , \upsilon )$, A\ref{assu:tv} yields a stronger requirement on the sensitivity of the distribution shift than A\ref{assu:w1} in general. 
Moreover, we require that
\begin{Assumption}\label{assu:bd_loss}
    There exists a constant $\ell_{max} \geq 0$ such that
    $\sup_{\prm\in\RR^d, z\in {\sf Z}} | \ell(\prm; z) | \leq \ell_{max}$.
\end{Assumption}
The above condition can be satisfied in practical scenarios where \eqref{eq:p} involves the training of nonlinear models such as neural networks with bounded outputs.

Under the above conditions, we observe the following upper bound on the SPS measure in Definition~\ref{def:sps} for the clipped SGD algorithm:
\begin{theorem} \label{thm:nvcx}
    Under A\ref{assu:lips}, \ref{assu:bndgrd}, \ref{assu:var-ncvx}, \ref{assu:tv}, \ref{assu:bd_loss}. Let the step sizes satisfy $\sup_{t\geq 1} \gamma_{t} \leq \frac{1}{2(1+\sigma_1^2)}$. Then, for any $T\geq 1$, the iterates $\{ \prm_{t} \}_{t\geq 0}$ generates by \eqref{algo:pcsgd} satisfy:
    \begin{align}
    \sum_{t=0}^{T-1} \gamma_{t+1} \EE \left[ \norm{\grd f(\prm_t; \prm_t)}^2 \right] & \leq 8 \Delta_0 + 4L\sigma_0^2 \sum_{t=0}^{T-1}\gamma_{t+1}^2  + 8 {\sf b}(\beta,c) \sum_{t=0}^{T-1}\gamma_{t+1}, \label{thmeq:nvcx-pcsgd}
    \end{align}
    where $\Delta_0 \eqdef \EE [f(\prm_0; \prm_0) - \ell^\star]$ is an upper bound to the initial optimality gap for performative risk, and
    \beq\notag
    {\sf b}(\beta,c) \! = \! \ell_{\max} \beta (\sigma_0 + 8(1+\sigma_1^2) \ell_{\max} \beta) \! + \! 2\max\{G \! - \! c, 0\}^2. 
    \eeq
\end{theorem}
The proof is relegated to \S\ref{app:noncvx}.

To get further insights, fix any $T \geq 1$ and set a constant step size $\gamma_{t} = 1/\sqrt{T}$, we let ${\sf T}$ be a random variable chosen uniformly and independently from $\{0,1,\cdots, T-1\}$. The iterates satisfy:
\begin{align} \notag
\EE \left[ \norm{\grd f(\prm_{\sf T}; \prm_{\sf T})}^2 \right] &\leq \left( \Delta_{0} + \frac{L\sigma_0^2}{2} \right) \frac{8}{\sqrt{T}} + 8 \, {\sf b}(\beta,c) .
\end{align}
We observe that the first term vanishes as $T \to \infty$. The second term represents an upper bound to the asymptotic bias for the clipped SGD algorithm that scales with the distribution shift's sensitivity $\beta$ and the clipping threshold $c$. Notice that as shown in \citep{koloskova2023revisiting}, the clipped SGD algorithm admits a non-zero asymptotic bias for the case of non-convex optimization. In comparison, our bound can be directly controlled by the clipping threshold.

Compared to the findings for the strongly convex case, while the asymptotic bias persists in Theorem~\ref{thm:nvcx} and it also depends on $\beta$ and $\max\{G-c,0\}$, the latter effects are combined in an \emph{additive} fashion. Nevertheless, we suspect that this bound on the bias can be improved in the non-convex case. We note that in general, finding a tight lower bound for the convergence of clipped SGD in the non-convex, decision-dependent setting is an open problem.

\subsection{Differential Privacy Guarantees} \label{sec:privacy}
Our next objective is to study the implications of the convergence analysis on the privacy preservation power of {\algoname}.
To fix idea, we first introduce the definition of \emph{$(\varepsilon, \delta)$ differential privacy (DP)} measure which is customary for measuring the level of privacy leakage of a stochastic algorithm:
\begin{definition}\label{def:privacy} \citep{DF_algoithmic_foundation}
A randomized mechanism ${\cal M}:{\cal D} \mapsto {\cal R}$ satisfies $(\varepsilon,\delta)-$differential privacy if for any two adjacent inputs ${\sf D}, {\sf D}^\prime \in {\cal D}$ which differs by only 1 different sample, and for any subset of outputs $S\subseteq {\cal R}$,
\beq 
    \Pr[{\cal M}( {\sf D} ) \in S] \leq e^\varepsilon \Pr[{\cal M}( {\sf D}^\prime) \in S] + \delta.
\eeq 
\end{definition}
The definition can be understood through the case with $\varepsilon, \delta \approx 0$. In such case, the output (which consists of the entire training history) of an DP algorithm on two adjacent databases, ${\sf D}_0, {\sf D}_0'$ will be indistinguishable, thus protecting the identity of each data sample.

To discuss privacy preservation in the framework of \citep{abadi2016deep} using Definition~\ref{def:privacy}, we need to introduce a few specifications on the performative prediction problem \eqref{eq:p} and modifications to the {\pcsgd} algorithm. First, we consider a \emph{fixed, finite database} setting for \eqref{eq:p}. The database consists of $m$ samples ${\sf D}_0 := \{ \bar{z}_i \}_{i=1}^m$. The  decision-dependent distribution  ${\cal D}(\prm)$ is defined accordingly: $Z \sim {\cal D}(\prm)$ refers to a (batch of) sample drawn from the database ${\sf D}_0$ with \emph{distribution shift} governed by a random map $\mathcal{S}_i: {\cal X} \rightarrow {\sf Z}$. Let $s_i (\prm)$ be a realization of ${\cal S}_i(\prm)$, 
\beq \label{eq:shift_data}
Z = \bar{z}_i + s_i(\prm ), \quad i \sim {\rm Unif}( [m] ).
\eeq 
Second, we concentrate on a slightly modified version of \eqref{algo:pcsgd} proposed in \citep{abadi2016deep}: set $Z_{t+1} \sim {\cal D}(\prm_t)$,
\begin{align}\label{algo:pcsgd-dp}
    \hspace{-.26cm} \prm_{t+1}  = {\cal P}_{\cal X} \big(\prm_t  \!-\!  \gamma_{t+1} ( \clip_c( \grd \ell(\prm_t; Z_{t+1})) \!+\! \zeta_{t+1} ) \big),
\end{align}
for $t=0,...,T-1$,
where $\zeta_{t+1} \sim {\cal N}( {\bm 0}, \sigmaDP{2}{\bm I})$ is an artificial (Gaussian) noise added to preserve privacy. Compared to approaches such as \citep{chaudhuri2011differentially} which directly add Laplacian noise to SGD, the algorithm with clipping offers better numerical stability.  

Overall, we observe that \eqref{algo:pcsgd-dp} is \emph{a randomized mechanism} applied on the \emph{fixed dataset} ${\sf D}_{0}$, where the distribution shifts is treated as a part of the mechanism. In fact, we can analyze the DP measure of {\pcsgd} using the following corollary of 
\citep{abadi2016deep}:
\begin{corollary}\label{thm:dp}(Privacy Guarantee) 
For any $\varepsilon \leq T/m^2$, $\delta \in (0, 1)$, and $c > 0$, the {\pcsgd} algorithm with greedy deployment is $(\varepsilon, \delta)$-DP after $T$ iterations if we let
$\sigmaDP{} = {c \sqrt{T \log(1/\delta)}} / ({m \varepsilon})$.    
\end{corollary}
See 
\S\ref{app:privacy} for detailed proof.
Corollary \ref{thm:dp} states that {\pcsgd} can achieve $(\varepsilon, \delta)$-DP with an appropriate DP noise level. 
Furthermore, the $\sqrt{T}$ dependence for $\sigmaDP{}$ leads to an additional source of bias in \eqref{eq:thm1}.
Together with Theorem~\ref{thm1}, the corollary gives a guideline for setting the algorithm's parameters such as clipping threshold $c$ and step size $\gamma$.

We are now equipped with the machinery to study the implications of previous convergence results with respect to the DP guarantees for strongly convex $\ell(\cdot)$. Suppose that the DP parameters $(\varepsilon, \delta)$ are fixed. From Corollary~\ref{thm:dp}, the DP noise variance $\sigmaDP{2}$ is proportional to $T$. It follows that increasing $T$ leads to an increase in the error term $c_1$ observed in Theorem~\ref{thm1} and it adds to the bias even when $\gamma_t = {\cal O}(1/T)$. In this light, the next result demonstrates how to design an optimal constant step size that minimizes the upper bound in \eqref{eq:thm1} corresponding to the efficacy of the model:


\begin{corollary}{(Finite-time Analysis)}\label{cor:finite}
   Fix the privacy parameters at $(\varepsilon, \delta)$ and clipping threshold at $c$. Assume that $G>c$ and a constant step size is used in {\pcsgd}. {For all $T\geq 1$}, to achieve the minimum for $\EE \| \hat{\prm}_T \|^2$,
   the optimal constant step size can be set as
    \beq \label{eq:gamma_star}
    \textstyle \gamma^\star = \frac{\log \Delta(\tmu)^{-1}}{\tmu T},~\text{where}~\Delta(\tmu) \eqdef \frac{2(2(c^2+G^2)+d\sigmaDP{2})}{T\tmu^2 \normtxt{\hat{\prm}_0}^2}.
    \eeq
    Let $\phi\eqdef \frac{d \log (1/\delta)}{m^2 \varepsilon^2 }$,
    then (\ref{eq:thm1}) simplifies to
    \[
    \EE\norm{\hat{\prm}_T}^2 = {\cal O}\left( \frac{{\cal C}_1}{\tmu^2} + \left[ \frac{c^2 + G^2}{T \tmu^2} + \frac{c^2 \phi}{\tmu^2} \right] \log\left(\frac{\tmu^2}{\phi c^2}\right) \right).
    \]
\end{corollary}
Besides the ${\cal O}(1/T)$
dependence, we observe that $\gamma^\star$ is affected by the sensitivity parameter through $\tmu^2$, and the DP parameters $\varepsilon, \delta$.

Lastly, we examine the case when $T \gg 1$. Observe that by setting $\gamma_t = c_2 / \tmu T$ for any $c_2 > 1$, the first term in \eqref{eq:thm1} vanishes with sufficiently large $T$.
Similar to the above corollary, enforcing DP guarantee leads to an asymptotic bias in \eqref{eq:thm1} {which} depends on the interplay between {the} clipping threshold $c$, DP noise variance $\sigmaDP{2}$, etc. We obtain the following asymptotic guarantee upon optimizing the clipping threshold $c$: 
\begin{corollary}{(Asymptotic Analysis)}\label{cor:asym}
Fix the privacy parameters at $(\varepsilon, \delta)$. Let $T \gg 1$ and set $\gamma_t = (1+c_2) / (\tmu T)$, 
the optimum asymptotic upper bound for the deviation from $\prm_{PS}$ in Theorem \ref{thm1} is given by
\begin{align}
    &\EE\norm{ {\prm}_\infty - \prm_{PS} }^2 = {\cal O}\left( \frac{G^2}{\tmu^2} (1+\frac{d\log(1/\delta)}{ m^2 \varepsilon^2}) \right).
\end{align}
which is achieved by setting the clipping threshold as
\[ 
c^\star = \frac{2 G m^2 \varepsilon^2}{ d \log(1/\delta) + 2 m^2 \varepsilon^2}.
\]
\end{corollary}
We observe that the asymptotic deviation from $\prm_{PS}$ is in the order of ${\cal O}( d / ( \tmu^2 \varepsilon^2 ) )$ which contains the combined effects from the sensitivity of distribution shift and DP requirement.

\section{Reducing Clipping Bias in Clipped SGD}\label{sec:dice}
The previous section illustrates that {\pcsgd} suffers from a \emph{bias amplification} phenomenon in the setting of performative prediction. 
Although the issue can be remedied by tuning the step size $\gamma$ and clipping threshold $c$, as in Corollaries~\ref{cor:finite}, \ref{cor:asym}, it may not be feasible for practical applications as the problem parameters such as $\tmu, G$ may be unknown. 

In this section, we discuss how to reduce the clipping bias through applying a recently proposed clipped SGD algorithm {\dicesgd} \citep{zhang2023differentially}. The latter is proven to achieve DP while converging to the exact solution of a convex optimization problem. We adapt this algorithm in the performative prediction setting and show that it removes the clipping bias inflicted by {\pcsgd}.

For the subsequent discussion, we consider the unconstrained setting ${\cal X} \equiv \RR^d$. The {\dicesgd} algorithm is summarized in Algorithm \ref{algo:dicesgd} with the greedy deployment mechanism for the performative prediction setting. 
Compared to {\pcsgd}, the notable differences include the use of \emph{two} clipping operators in line~4 for forming the stochastic gradient estimate $v_{t+1}$, and an \emph{error feedback} step in line~5 where $e_t$ accumulates the error due to clipping. The vector $e_t$ is  a private variable kept by the learner. 

We remark that the pseudo code describes a general implementation which includes the Gaussian noise mechanism for privacy protection. As shown in \cite{zhang2023differentially}, the use of two clipping operators reduces the privacy leakage. With an appropriate $\sigmaDP{2}$, the algorithm is guaranteed to achieve $(\varepsilon, \delta)$-Renyi DP, a relaxed notion for DP proposed in \cite{mironov2017renyi}. 
Nonetheless, when the DP requirement is not needed, one may set $\sigmaDP{2} = 0$ for a clipped algorithm with reduced bias.

\begin{algorithm}[tb]
   \caption{{\dicesgd} with Greedy Deployment}
   \label{algo:dicesgd}
\begin{algorithmic}[1]
   \STATE {\bfseries Input:} $C_1$, $C_2$, $a_0, a_1$, $\varepsilon$, $\delta$, {${\sf D}_0$}, $\sigmaDP{2}$ with $C_2 \geq C_1$, initialization $\prm_0$, $e_0 = 0$.
   \FOR{$t=0 \textbf{ to } T-1$}
    \STATE Draw new sample $Z_{t+1}\sim {\cal D}(\prm_t)$ and Gaussian noise $\zeta_{t+1} \sim {\cal N}(0, \sigmaDP{2} \mathbf{I})$.
    \STATE $v_{t+1} = \clip_{C_1}(\grd \ell(\prm_t; Z_{t+1})) + \clip_{C_2}(e_t)$.
    \STATE $\prm_{t+1}  =  \prm_t - \gamma_{t+1} (v_{t+1} + \zeta_{t+1}),$  
    
    $e_{t+1} = e_t + \grd \ell(\prm_t; Z_{t+1}) - v_{t+1}$.
   \ENDFOR
   \STATE {\bfseries Output:} Last iterate $\prm_{T}$.
\end{algorithmic}
\end{algorithm} 


Importantly, the error feedback mechanism is effective in removing the asymptotic bias. To see that this insight can be extended to the performative prediction setting, observe that any fixed point $(\bar{e}, \bar{\prm})$ of Algorithm~\ref{algo:dicesgd} satisfies
\beq 
\begin{aligned}
-\clip_{C_2}( \bar{e} ) & = \EE_{ Z \sim {\cal D}(\bar{\prm}) } [ \clip_{C_1}( \grd \ell( \bar{\prm}; Z) ) ] \\
\grd f( \bar{\prm}; \bar{\prm} ) - \clip_{C_2}( \bar{e} ) & = \EE_{ Z \sim {\cal D}(\bar{\prm}) } [ \clip_{C_1}( \grd \ell( \bar{\prm}; Z) ) ]
\end{aligned}
\eeq 
Under the condition $C_2 \geq C_1$, a feasible fixed point $(\bar{e}, \bar{\prm})$ shall satisfy 
\[
    \grd f( \bar{\prm}; \bar{\prm} ) = {\bm 0} \text{ and } \bar{e} = - \EE_{ Z \sim {\cal D}(\bar{\prm}) } [ \clip_{C_1}( \grd \ell( \bar{\prm}; Z) ) ]
\]
since 
\beq \label{eq:fixedpt}
\| \bar{e} \| \leq \EE_{ Z \sim {\cal D}(\bar{\prm}) } [ \|\clip_{C_1}( \grd \ell( \bar{\prm}; Z) ) \| ] \leq C_1,
\eeq 
where the first inequality is due to Jensen's inequality.
The condition $\grd f( \bar{\prm}; \bar{\prm} ) = {\bm 0}$ implies  $\bar{\prm} = \prm_{PS}$ {under strongly-convex $\ell(\cdot)$}. 

We conclude by presenting the convergence for the {\dicesgd} algorithm in the performative prediction setting.  We first observe the assumption: 
\begin{Assumption}\label{assu:et}
There exists a constant $M$ such that for any $t\geq 1$, 
$\EE [\norm{e_t}^2 ] \leq M^2$.
\end{Assumption}
A\ref{assu:et} is verified empirically in our experiments for the case of $C_2 \geq C_1$; see \S\ref{app:addexp}.
Similar to {\pcsgd}, the {\dicesgd} algorithm is also a non-gradient algorithm with non-smooth drifts. The following analysis is achieved by designing a suitable Lyapunov function for each type of $\ell(\cdot)$.

\paragraph{Strongly Convex Loss}
Notice that A\ref{assu:var-ncvx} implies that there exists $G,B \geq 0$ with
\beq \label{eq:gb_bound}
    \EE_{Z \sim {\cal D}(\prm)} [ \| \grd \ell( \prm; Z ) \|^2 ] \leq G^2 + B^2 \| \prm - \prm_{PS} \|^2,
\eeq 
for any $\prm\in \RR^d$.
Denote  $\tmu \eqdef \mu - L\beta$ and $\Bprm_{t} := \prm_t - \gamma_t e_t - \prm_{PS}$, we have
\begin{theorem} \label{thm:dicesgd}
Under A\ref{ass:scvx}, \ref{assu:lips}, \ref{assu:var-ncvx}, \ref{assu:w1}, \ref{assu:et} and \eqref{eq:gb_bound} holds. Suppose that $\beta < \frac{\mu}{L}$, there exists $a_0, a_1,b,\bar{b} \geq 0$ such that the step size {of {\dicesgd}} satisfies i) $\gamma_t = a_0/(a_1 + t)$, $a_0 \geq 1 / b$,
   ii) $\gamma_t \leq \min \{ 
   {\tmu}/({16 b}),
   {\blue{8}} / {\tmu}, { \tmu} / {4B^2} \}$, and iii)   ${\gamma_t^2} / {\gamma_{t+1}^2} \leq 1+ \bar{b} \gamma_{t+1}^2$,
Then, for any $t \geq 0$,
\begin{align}
 \EE\norm{\Bprm_{t+1}}^2 &  \leq \prod_{i=1}^{t+1}\left( 1 - \frac{\tmu}{\blue{4}} \gamma_{i} \right)\norm{\Bprm_{0}}^2 + \frac{\blue{8} (G^2 + d\sigmaDP{2})}{\tmu}\gamma_{t+1} \label{eq:dicesgd_conv} 
 \\
 &\quad + \frac{\blue{16}L^2M^2(1+\beta)^2}{\tmu^2}\gamma_{t+1}^2  + \frac{\blue{24} b^2 M^2}{\tmu} \gamma_{t+1}^3  + \frac{\blue{16} L^2M^2 \bar{b}(1+\beta)^2}{\tmu^2} \gamma_{t+1}^4. \notag
\end{align}
\end{theorem}
We present the detailed proof in \S\ref{app:dice}. 
Under strong convex losses A\ref{ass:scvx}, our analysis shows that {\dicesgd} converges to a unique fixed point ($\prm_{PS}$) in the mean square sense. Note that the analysis framework here differs significantly from that for non-convex losses.

We observe from \eqref{eq:dicesgd_conv} that the dominant term on the right-hand-side is the second term which behaves as ${\cal O}( \gamma_{t+1} ) = {\cal O}(1/t)$. Consequently, we have  \vspace{-.1cm}
\[
\EE \| \prm_t - \prm_{PS} \|^2  \leq 2 \EE\norm{\Bprm_{t}}^2 + 2 \gamma_t^2 \EE \norm{e_t}^2 = {\cal O}(1/t),
\]
due to A\ref{assu:et}. In other words, the {\dicesgd} algorithm is asymptotically \emph{unbiased}.

\paragraph{Non-convex Loss}
Our last endeavor is to show that for the non-convex loss setting, {\dicesgd} also reduces the bias due to clipping {in performative prediction settings}. We observe the convergence result:
\begin{theorem}\label{thm:ncvx-dice}
Under A\ref{assu:lips}, \ref{assu:var-ncvx}, \ref{assu:tv}, \ref{assu:bd_loss}, \ref{assu:et}. 
If we set $\gamma = 1/\sqrt{T}$, then for sufficiently large $T$, it holds
\begin{equation} \label{eq:dicesgd_noncvx}
    \min_{t=0,...,T-1} \EE[ \| \grd f( \prm_t; \prm_t ) \|^2 ] = {\cal O} \left( \frac{ 1 }{ \sqrt{T} } + {\sf b} \beta \right) ,
\end{equation}
where ${\sf b} = {\cal O}( \ell_{\max} ( (C_1+C_2) + \sqrt{d} \sigmaDP{2} ) )$.
\end{theorem}
The detailed theorem and proof can be found in \S\ref{app:dicesgd-ncvx}. Our analysis involved a few modifications over \citep[Theorem 3.6]{zhang2023differentially} for the decision-dependent distribution. 

Unlike the convergence Theorems \ref{thm:dp} \& \ref{thm:nvcx} for the {\pcsgd} algorithm, the analysis of {\dicesgd} relaxes the uniformly bounded gradient assumption (A\ref{assu:bndgrd}) to the variance-based assumption (A\ref{assu:var-ncvx}). This relaxation is enabled by the feedback mechanism embedded in the {\dicesgd} algorithm.

Importantly, from \eqref{eq:dicesgd_noncvx}, we observe that when $\beta \approx 0$, there is no asymptotic bias for {\dicesgd}. However, we also note that unlike {\pcsgd}, the multiplicative factor ${\sf b}$ depends on $\sigmaDP{2}, C_1, C_2$ which indicates that the bias due to distribution shift may become more sensitive {when using} {\dicesgd}.



\begin{remark} 
    The DP guarantees of {\dicesgd} with distribution shift can be studied by extending \citep[Theorem 3.7 \& Appendix A.2]{zhang2023differentially}. Particularly, we can model the distribution shift through a random mapping $f: (Z + D_0) \mapsto D_0$, similar to the one introduced in Appendix \ref{app:privacy}. Applying the data processing inequality shows that a comparable bound for the DP guarantees of the {\dicesgd} algorithm under distribution shift to \citep[Theorem 3.7]{{zhang2023differentially}} can be established.    
\end{remark}

\vspace{-.1cm}
\section{Numerical Experiments}\label{sec:experiment}
All experiments are performed with Python on a server using a single Intel Xeon 6138 CPU thread.
In the interest of space, we only consider the experiments with strongly convex $\ell(\cdot)$ and focus on the setting with DP guarantees to validate our theoretical findings.
For the sake of fair comparison between the {\pcsgd} and the {\dicesgd} algorithms, we choose the set ${\cal X}$ to be such that the optimal point of the unconstrained and constrained case of \eqref{def:proj_thetaps} will coincide. 
To maintain the same DP guarantees, we respectively set the DP noise standard deviation for {\pcsgd} and {\dicesgd} as $\sigmaDP{}$ and $\sqrt{96}\sigmaDP{}$, according to 
Corollary \ref{thm:dp} and \citep[Theorem 3.7]{zhang2023differentially}. 
For the {\dicesgd} algorithm, we set $C_1 = C_2$. This decision is motivated by the necessity to maintain a balanced trade-off in Algorithm \ref{algo:dicesgd}, where augmenting the values of $C_1$ and $C_2$ would entail an increase in the variance of the Gaussian mechanism in line 3. 
\begin{figure*}[t]
    \centering
    \includegraphics[width=.3\linewidth]{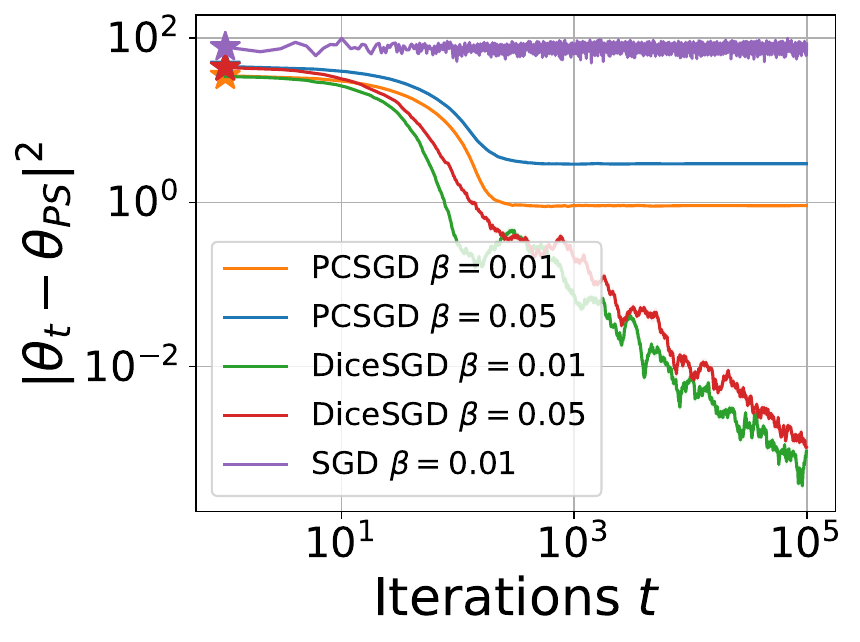}
    \includegraphics[width=.3\textwidth]{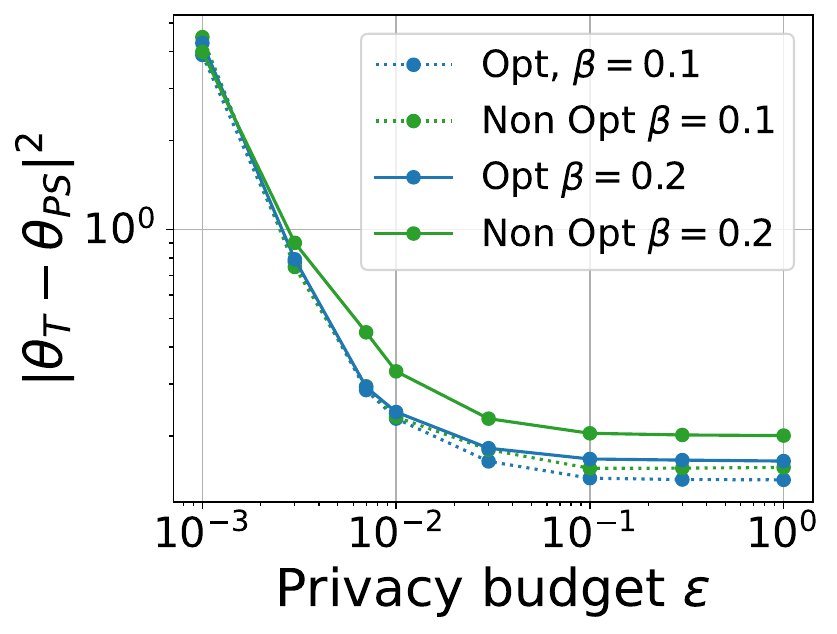}
    \includegraphics[width=.3\textwidth]{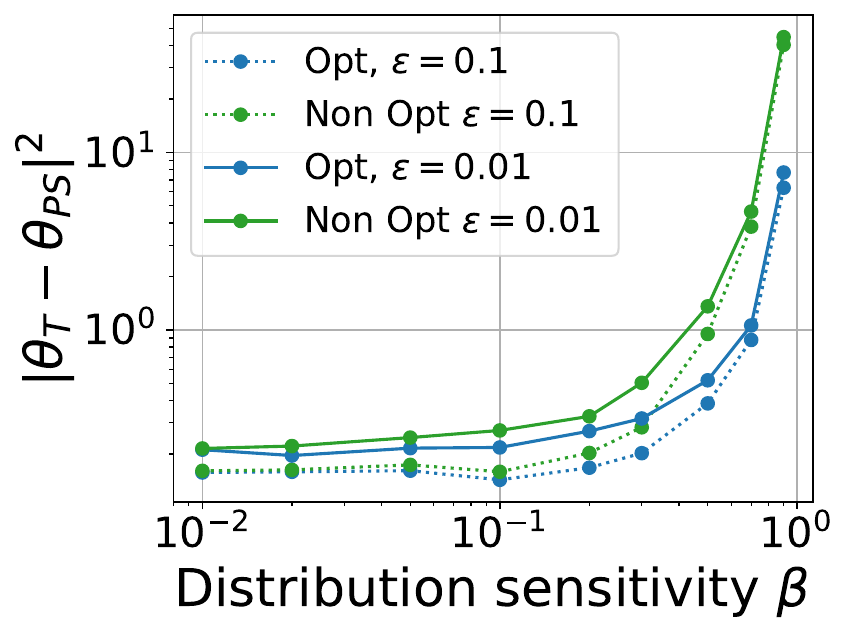}
    \vspace{-.3cm}
    \caption{{\bf Quadratic Minimization} (First) The performative stability gap $\norm{\prm_t - \prm_{PS}}^2$. (Second) Trade off between privacy budget $\varepsilon$ and bias. (Third) Bias amplification effect due to  $\beta$. 
    }\vspace{-.2cm}
    \label{fig1}
\end{figure*}

\begin{figure*}[t]
    \centering
    \includegraphics[width=.3\textwidth]{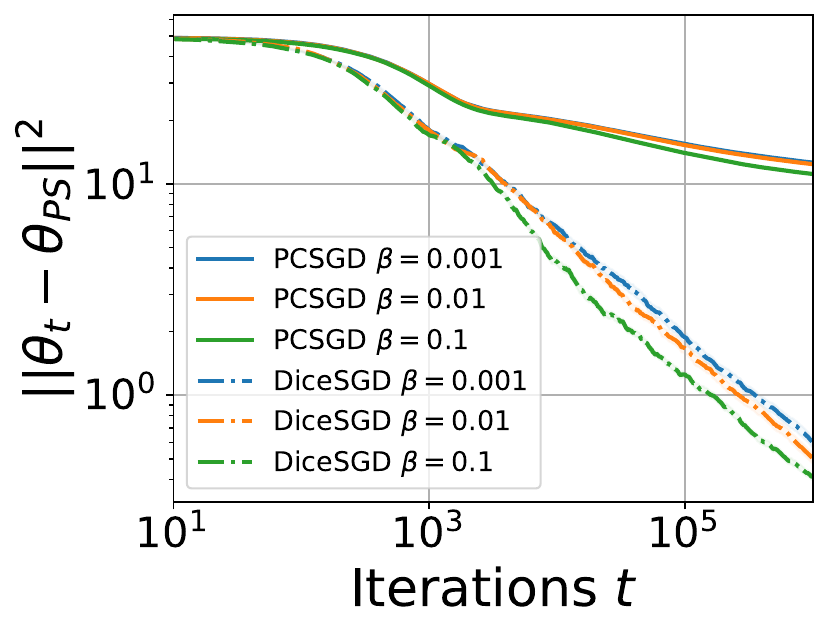}
    \includegraphics[width=.3\textwidth]{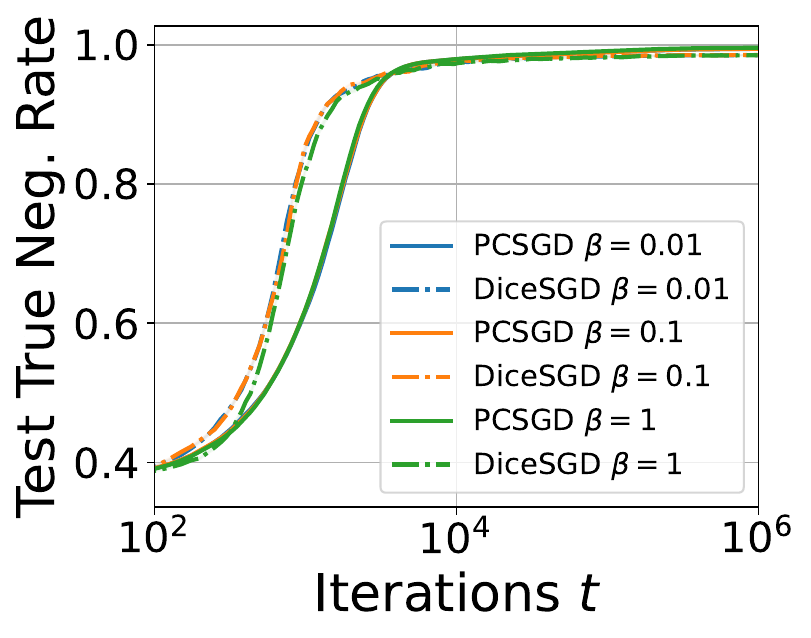}
    \includegraphics[width=.3\textwidth]{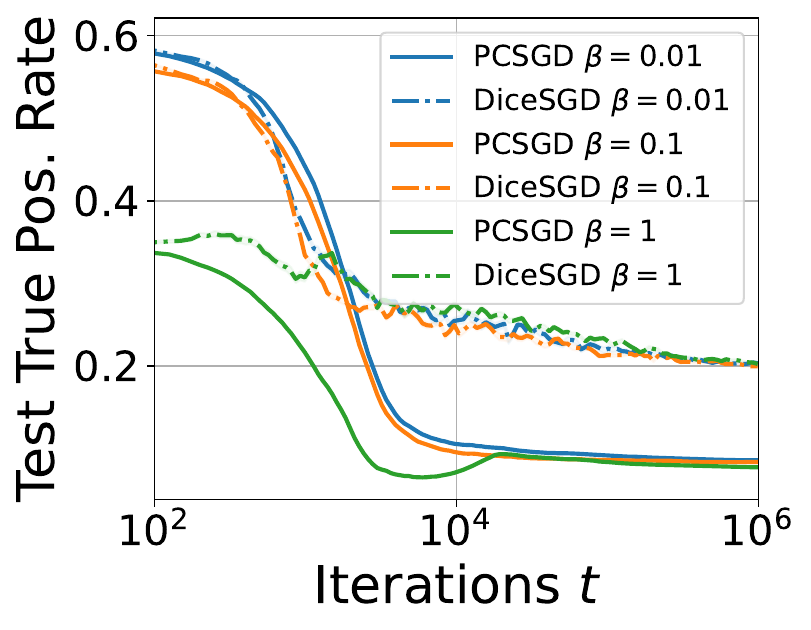}
    \vspace{-.3cm}
    \caption{{\bf Logistic Regression} (First) Gap between iterations and performative stable point $\norm{\prm_t - \prm_{PS}}^2$.  (Second) Test true negative rate with shifted distribution. (Third) Test true positive accuracy with shifted distribution.
    }\vspace{-.3cm}
    \label{fig2}
\end{figure*}

\vspace{+.3cm}
\noindent {\bf Quadratic Minimization.} The \emph{first problem} is concerned with the validation of Theorems \ref{thm1}, \ref{thm:lb}, and \ref{thm:dicesgd}. Here, we consider a scalar performative risk optimization problem with synthetic data
\[ 
\min_{\prm \in {\cal X}}~ \EE_{z\sim {\cal D}(\prm)}[(\prm + az)^2/2],
\]
where ${\cal D}(\prm)$ is a uniform distribution over the data points $\{ b \tilde{Z}_i - \beta \prm \}_{i=1}^m$ such that $\tilde{Z}_i \sim {\cal B}(p)$ is Bernoulli and $a>0, b>0$, $p<1/2$. We also set ${\cal X} = [-10, 10]$ and observe that for $0 < \beta < a^{-1}$, the performative stable solution is $\prm_{PS} = -\frac{\bar{p}a}{1-a\beta}$, where $\bar{p} = \frac{1}{m} \sum_{i=1}^m \tilde{Z}_i$ is the sample mean.

We set $p=0.1, \varepsilon=0.1, 
{\delta = 1/m}, \beta \in \{0.01, 0.05\}, a=10, b=1, c= C_1=C_2=1$, the sample size {$m = 10^5$}.
The step size is $\gamma_t = \frac{10}{100+t}$ with the initialization $\prm_0= 5$. In Fig.~\ref{fig1} ({first plot}), we compare $| \prm_t - \prm_{PS} |^2$ against the iteration number $t$ 
using plain {\bf SGD} with DP noise, {\pcsgd} and {\dicesgd}. As observed, adding the DP noise compromises {\bf SGD}'s convergence. {\pcsgd} cannot converge to $\prm_{PS}$ due to the clipping bias which increases as $\beta \uparrow$. Meanwhile, {\dicesgd} finds a \emph{bias-free} solution as it converges to $\prm_{PS}$ at rate ${\cal O}(1/t)$. 

Our next experiments examine the trade-off between clipping bias of {\pcsgd} $| \prm_T - \prm_{PS} |^2$ and privacy budget $\varepsilon$ or distribution sensitivity $\beta$. We set $ a=1, b=6, c=c^\star\approx 2.32$, $T=10^5$, $\beta \in \{0.1, 0.2\}$ or $\varepsilon \in \{0.01, 0.1\}$, while keeping the other parameters unchanged. Using Corollary~\ref{cor:finite}, we set the optimal step size
according to $\gamma^\star$ in \eqref{eq:gamma_star},
and the non-optimal step size as $\gamma= \frac{\log (1/\Delta(\mu))}{\mu T}$ to simulate the scenario when the presence of distribution shift is unknown. From Fig.~\ref{fig1} ({second} \& {third} plots), setting the optimal step size $\gamma^\star$ adapted to distribution shifts achieves a smaller bias in all settings. Meanwhile, as the privacy budget \blue{decreases} $\varepsilon \downarrow 0$ or the sensitivity of distribution shift \blue{increases} $\beta \uparrow \frac{\mu}{L}$, the bias of {\pcsgd} increases.  \vspace{+.3cm}

\noindent {\bf Logistic Regression.} We consider the real dataset {\tt GiveMeSomeCredit} \cite{kaggle2011give} with $m=15776$ samples and $d=10$ features. We split the training/test sets using the ratio of $7:3$. The learner aims to find a classifier via minimizing the regularized logistic loss:
\begin{align} \notag 
\ell(\prm; z) = \alpha(z) \left( \log(1 \!+\! \exp(x^\top \prm)) \!-\! y x^\top \prm \right) \!+\! \frac{\eta}{2} \norm{\prm}^2,
\end{align}
where $\eta = 10^2/m$ is a regularization parameter, $z \equiv (x,y) \in \RR^d \times \{0,1\}$ is the training sample, {$\alpha(z)=y+1$ is a label weight}, $y=0$ $(y=1)$ denotes a customer without (with) history of defaults. The strategic behavior of the population, i.e., their features $x$ are adapted to $\prm$ through maximizing a quadratic utility function.

We fix the privacy budget at $\varepsilon=1$, $\delta = 1/m$, clipping thresholds {$c=C_1=C_2=1$, step size $\gamma_{t}=50/(5000+t)$}, and sensitivity parameter $\beta\in \{0.001, 0.01, 0.1\}$. From Fig.~\ref{fig2} (first plot), we observe the gap $\norm{\prm_t - \prm_{PS}}^2$ of {\dicesgd} decays at ${\cal O}(1/t)$ as $t\rightarrow\infty$, {\pcsgd} achieve the steady state due to bias after the transient stage. This coincides with Theorems \ref{thm1} \& \ref{thm:dicesgd}. In Fig. \ref{fig2} (second \& third plots), we compare the test accuracy against the iteration number $t$.
Here, the trajectory becomes more unstable as $\beta \uparrow$. 
An interesting observation is that increasing the sensitivity $\beta$ leads to lower true positive rate.

\paragraph{Conclusions.} This paper initiates the study of clipped SGD algorithms in the performative prediction setting. In both cases with strongly convex and non-convex losses, we discovered a \emph{bias amplification} effect with the {\pcsgd} algorithm and proposed several remedies including an extension of the {\dicesgd} algorithm to performative prediction. 


\bibliographystyle{plainnat}
\bibliography{ref.bib}

\begin{thebibliography}{30}
\providecommand{\natexlab}[1]{#1}
\providecommand{\url}[1]{\texttt{#1}}
\expandafter\ifx\csname urlstyle\endcsname\relax
  \providecommand{\doi}[1]{doi: #1}\else
  \providecommand{\doi}{doi: \begingroup \urlstyle{rm}\Url}\fi

\bibitem[Abadi et~al.(2016)Abadi, Chu, Goodfellow, McMahan, Mironov, Talwar,
  and Zhang]{abadi2016deep}
Martin Abadi, Andy Chu, Ian Goodfellow, H~Brendan McMahan, Ilya Mironov, Kunal
  Talwar, and Li~Zhang.
\newblock Deep learning with differential privacy.
\newblock In \emph{Proceedings of the 2016 ACM SIGSAC conference on computer
  and communications security}, pages 308--318, 2016.

\bibitem[Brown et~al.(2022)Brown, Hod, and Kalemaj]{brown2022performative}
Gavin Brown, Shlomi Hod, and Iden Kalemaj.
\newblock Performative prediction in a stateful world.
\newblock In \emph{International Conference on Artificial Intelligence and
  Statistics}, pages 6045--6061. PMLR, 2022.

\bibitem[Chaudhuri et~al.(2011)Chaudhuri, Monteleoni, and
  Sarwate]{chaudhuri2011differentially}
Kamalika Chaudhuri, Claire Monteleoni, and Anand~D Sarwate.
\newblock Differentially private empirical risk minimization.
\newblock \emph{Journal of Machine Learning Research}, 12\penalty0 (3), 2011.

\bibitem[Chen et~al.(2020)Chen, Wu, and Hong]{chen2020understanding}
Xiangyi Chen, Steven~Z Wu, and Mingyi Hong.
\newblock Understanding gradient clipping in private sgd: A geometric
  perspective.
\newblock \emph{Advances in Neural Information Processing Systems},
  33:\penalty0 13773--13782, 2020.

\bibitem[Drusvyatskiy and Xiao(2023)]{drusvyatskiy2023stochastic}
Dmitriy Drusvyatskiy and Lin Xiao.
\newblock Stochastic optimization with decision-dependent distributions.
\newblock \emph{Mathematics of Operations Research}, 48\penalty0 (2):\penalty0
  954--998, 2023.

\bibitem[Dwork and Roth(2014)]{DF_algoithmic_foundation}
Cynthia Dwork and Aaron Roth.
\newblock The algorithmic foundations of differential privacy.
\newblock \emph{Foundations and Trends in Theoretical Computer Science},
  9\penalty0 (3–4):\penalty0 211--407, 2014.

\bibitem[Gorbunov et~al.(2020)Gorbunov, Danilova, and
  Gasnikov]{gorbunov2020stochastic}
Eduard Gorbunov, Marina Danilova, and Alexander Gasnikov.
\newblock Stochastic optimization with heavy-tailed noise via accelerated
  gradient clipping.
\newblock \emph{Advances in Neural Information Processing Systems},
  33:\penalty0 15042--15053, 2020.

\bibitem[Hardt and Mendler-D{\"u}nner(2023)]{hardt2023performative}
Moritz Hardt and Celestine Mendler-D{\"u}nner.
\newblock Performative prediction: Past and future.
\newblock \emph{arXiv preprint arXiv:2310.16608}, 2023.

\bibitem[Hardt et~al.(2016)Hardt, Megiddo, Papadimitriou, and
  Wootters]{hardt2016strategic}
Moritz Hardt, Nimrod Megiddo, Christos Papadimitriou, and Mary Wootters.
\newblock Strategic classification.
\newblock In \emph{Proceedings of the 2016 ACM conference on innovations in
  theoretical computer science}, pages 111--122, 2016.

\bibitem[Hazan et~al.(2015)Hazan, Levy, and Shalev-Shwartz]{hazan2015beyond}
Elad Hazan, Kfir Levy, and Shai Shalev-Shwartz.
\newblock Beyond convexity: Stochastic quasi-convex optimization.
\newblock \emph{Advances in neural information processing systems}, 28, 2015.

\bibitem[Izzo et~al.(2021)Izzo, Ying, and Zou]{izzo2021learn}
Zachary Izzo, Lexing Ying, and James Zou.
\newblock How to learn when data reacts to your model: performative gradient
  descent.
\newblock In \emph{International Conference on Machine Learning}, pages
  4641--4650. PMLR, 2021.

\bibitem[Kaggle(2011)]{kaggle2011give}
Kaggle.
\newblock Give me some credit, {{Improve}} on the state of the art in credit
  scoring by predicting the probability that somebody will experience financial
  distress in the next two years., 2011.
\newblock URL \url{https://www.kaggle.com/c/GiveMeSomeCredit}.

\bibitem[Khirirat et~al.(2023)Khirirat, Gorbunov, Horv{\'a}th, Islamov, Karray,
  and Richt{\'a}rik]{khirirat2023clip21}
Sarit Khirirat, Eduard Gorbunov, Samuel Horv{\'a}th, Rustem Islamov, Fakhri
  Karray, and Peter Richt{\'a}rik.
\newblock Clip21: Error feedback for gradient clipping.
\newblock \emph{arXiv preprint arXiv:2305.18929}, 2023.

\bibitem[Koloskova et~al.(2023)Koloskova, Hendrikx, and
  Stich]{koloskova2023revisiting}
Anastasia Koloskova, Hadrien Hendrikx, and Sebastian~U Stich.
\newblock Revisiting gradient clipping: Stochastic bias and tight convergence
  guarantees.
\newblock In \emph{International Conference on Machine Learning}, pages
  17343--17363. PMLR, 2023.

\bibitem[Li and Wai(2022)]{li2022state}
Qiang Li and Hoi-To Wai.
\newblock State dependent performative prediction with stochastic
  approximation.
\newblock In \emph{International Conference on Artificial Intelligence and
  Statistics}, pages 3164--3186. PMLR, 2022.

\bibitem[Li and Wai(2024)]{li2024stochastic}
Qiang Li and Hoi-To Wai.
\newblock Stochastic optimization schemes for performative prediction with
  nonconvex loss.
\newblock \emph{arXiv preprint arXiv:2405.17922}, 2024.

\bibitem[Li et~al.(2022)Li, Yau, and Wai]{li2022multi}
Qiang Li, Chung-Yiu Yau, and Hoi-To Wai.
\newblock Multi-agent performative prediction with greedy deployment and
  consensus seeking agents.
\newblock \emph{Advances in Neural Information Processing Systems},
  35:\penalty0 38449--38460, 2022.

\bibitem[Mai and Johansson(2021)]{mai2021stability}
Vien~V Mai and Mikael Johansson.
\newblock Stability and convergence of stochastic gradient clipping: Beyond
  lipschitz continuity and smoothness.
\newblock In \emph{International Conference on Machine Learning}, pages
  7325--7335. PMLR, 2021.

\bibitem[Mendler-D{\"u}nner et~al.(2020)Mendler-D{\"u}nner, Perdomo, Zrnic, and
  Hardt]{mendler2020stochastic}
Celestine Mendler-D{\"u}nner, Juan Perdomo, Tijana Zrnic, and Moritz Hardt.
\newblock Stochastic optimization for performative prediction.
\newblock \emph{Advances in Neural Information Processing Systems},
  33:\penalty0 4929--4939, 2020.

\bibitem[Mikolov et~al.(2012)]{mikolov2012statistical}
Tom{\'a}{\v{s}} Mikolov et~al.
\newblock Statistical language models based on neural networks.
\newblock \emph{Presentation at Google, Mountain View, 2nd April}, 80\penalty0
  (26), 2012.

\bibitem[Miller et~al.(2021)Miller, Perdomo, and Zrnic]{miller2021outside}
John~P Miller, Juan~C Perdomo, and Tijana Zrnic.
\newblock Outside the echo chamber: Optimizing the performative risk.
\newblock In \emph{International Conference on Machine Learning}, pages
  7710--7720. PMLR, 2021.

\bibitem[Mironov(2017)]{mironov2017renyi}
Ilya Mironov.
\newblock R{\'e}nyi differential privacy.
\newblock In \emph{2017 IEEE 30th computer security foundations symposium
  (CSF)}, pages 263--275. IEEE, 2017.

\bibitem[Mofakhami et~al.(2023)Mofakhami, Mitliagkas, and
  Gidel]{mofakhami2023performative}
Mehrnaz Mofakhami, Ioannis Mitliagkas, and Gauthier Gidel.
\newblock Performative prediction with neural networks.
\newblock In \emph{International Conference on Artificial Intelligence and
  Statistics}, pages 11079--11093. PMLR, 2023.

\bibitem[Narang et~al.(2023)Narang, Faulkner, Drusvyatskiy, Fazel, and
  Ratliff]{narang2023multiplayer}
Adhyyan Narang, Evan Faulkner, Dmitriy Drusvyatskiy, Maryam Fazel, and
  Lillian~J Ratliff.
\newblock Multiplayer performative prediction: Learning in decision-dependent
  games.
\newblock \emph{Journal of Machine Learning Research}, 24\penalty0
  (202):\penalty0 1--56, 2023.

\bibitem[Perdomo et~al.(2020)Perdomo, Zrnic, Mendler-D{\"u}nner, and
  Hardt]{perdomo2020performative}
Juan Perdomo, Tijana Zrnic, Celestine Mendler-D{\"u}nner, and Moritz Hardt.
\newblock Performative prediction.
\newblock In \emph{International Conference on Machine Learning}, pages
  7599--7609. PMLR, 2020.

\bibitem[Qui{\~n}onero-Candela et~al.(2022)Qui{\~n}onero-Candela, Sugiyama,
  Schwaighofer, and Lawrence]{quinonero2022dataset}
Joaquin Qui{\~n}onero-Candela, Masashi Sugiyama, Anton Schwaighofer, and Neil~D
  Lawrence.
\newblock \emph{Dataset shift in machine learning}.
\newblock Mit Press, 2022.

\bibitem[Shor(2012)]{shor2012minimization}
Naum~Zuselevich Shor.
\newblock \emph{Minimization methods for non-differentiable functions},
  volume~3.
\newblock Springer Science \& Business Media, 2012.

\bibitem[Zhang et~al.(2020{\natexlab{a}})Zhang, He, Sra, and
  Jadbabaie]{Zhang2020Why}
Jingzhao Zhang, Tianxing He, Suvrit Sra, and Ali Jadbabaie.
\newblock Why gradient clipping accelerates training: A theoretical
  justification for adaptivity.
\newblock In \emph{International Conference on Learning Representations},
  2020{\natexlab{a}}.

\bibitem[Zhang et~al.(2020{\natexlab{b}})Zhang, Karimireddy, Veit, Kim, Reddi,
  Kumar, and Sra]{zhang2020adaptive}
Jingzhao Zhang, Sai~Praneeth Karimireddy, Andreas Veit, Seungyeon Kim, Sashank
  Reddi, Sanjiv Kumar, and Suvrit Sra.
\newblock Why are adaptive methods good for attention models?
\newblock \emph{Advances in Neural Information Processing Systems},
  33:\penalty0 15383--15393, 2020{\natexlab{b}}.

\bibitem[Zhang et~al.(2024)Zhang, Bu, Wu, and Hong]{zhang2023differentially}
Xinwei Zhang, Zhiqi Bu, Zhiwei~Steven Wu, and Mingyi Hong.
\newblock Differentially private sgd without clipping bias: An error-feedback
  approach.
\newblock In \emph{ICLR}, 2024.

\end{thebibliography}

\newpage 
\appendix
\onecolumn


\section{Useful Facts}

In the proof of Theorems \ref{thm1} \& \ref{thm:dicesgd}, we utilize the following Lemma which is introduced in \citep[Lemma 6]{li2022multi}. The lemma upper bounds the coefficients of contraction equations in the form of \eqref{eq:contraction_proof_thm1} and \eqref{eq:contraction_eq_thm_dicesgd}:

\begin{lemma}{\citep[Lemma 6]{li2022multi}}\label{lem:aux}
Consider a sequence of non-negative, non-increasing step sizes $\{\gamma_{t}\}_{t \geq 1}$. Let $a>0$, $p\in \ZZ_+$ and $\gamma_{1}<2 / a$. If 
\[\left(\frac{\gamma_{t}}{\gamma_{t+1}}\right)^p \leq 1+\frac{a}{2}\cdot \gamma_{t+1}^p\] 
for any $t \geq 1$, then
\beq 
\sum_{j=1}^{t} \gamma_{j}^{p+1} \prod_{\ell=j+1}^{t}\left(1-\gamma_{\ell} a\right) \leq \frac{2}{a} \gamma_{t}^p,~~\forall~t \geq 1.
\eeq 
\end{lemma}
The proof of this lemma is presented in \citep[Lemma 6]{li2022multi} and is therefore omitted.

Additionally, we rely on the following lemma that provides smoothness guarantees regarding the performative prediction gradients:
\begin{lemma}\label{lem:continuity}
    Under A\ref{assu:lips}, \ref{assu:w1}. For any $\prm_0, \prm_1, \prm, \prm^\prime \in \RR^d$, it holds that
    \beq
        \norm{\grd f(\prm_0, \prm) - \grd f(\prm_1, \prm^\prime)} \leq L \norm{\prm_0 - \prm_1} + L\beta \norm{\prm - \prm^\prime}.
    \eeq
\end{lemma}
The proof of this lemma can be found in \citep[Lemma 2.1]{drusvyatskiy2023stochastic}.

\section{Proof of Theorem~\ref{thm1}}
\label{app:pf_scvx} 
We outline the main steps in proving the convergence for {\pcsgd}. In particular, we shall consider the general form of {\pcsgd} in \eqref{algo:pcsgd-dp} where the DP noise is introduced. To this end, we aim at proving the following bound: 
\begin{theorem*} Under A\ref{ass:scvx}, \ref{assu:lips}, \ref{assu:bndgrd}, \ref{assu:w1}. Suppose that $\beta <\frac{\mu}{L}$, the step sizes $\{\gamma_{t}\}_{t\geq 1}$ are non-increasing and satisfy
i) $\textstyle \frac{\gamma_{t-1}}{\gamma_{t}} \leq 1 + \frac{\mu -L\beta}{2}\gamma_t, $ and ii)
$\gamma_{t} \leq \frac{2}{\mu - L\beta}$.
Then, for any $t\geq 1$, the expected squared distance between $\prm_t$ and the performative stable solution $\prm_{PS}$ satisfies
\begin{align} \notag
    \hspace{-.2cm} \EE \normtxt{\hat{\prm}_{t+1}}^2  \leq  \prod_{i=1}^{t+1}(1- \tmu \gamma_{i}) \normtxt{\hat{\prm}_0}^2 + \frac{2 c_1}{\tmu} \gamma_{t+1} + \frac{8 {\cal C}_1}{\tmu^2},
\end{align}
where $c_1 \eqdef 2(c^2 + G^2) + d \sigmaDP{2}$, ${\cal C}_1 \eqdef (\max\{G-c, 0\})^2$, and $\tmu := \mu - L \beta$.
\end{theorem*}

To simplify notations, we define 
\beq 
\begin{aligned}
    \Tgrd g(\prm_{t}) &\eqdef \clip_{c} (\grd \ell(\prm_{t}; Z_{t+1})), \quad  b_t \eqdef \Tgrd g(\prm_t) - \grd f(\prm_t; \prm_{PS}).
\end{aligned}
\eeq 
Recall that $\hat{\prm}_t := \prm_t - \prm_{PS}$, the following lemma characterizes the one-step progress of {\pcsgd}.
\begin{lemma}\label{lem:onedes}
Under A\ref{assu:lips}, \ref{assu:bndgrd}, \ref{assu:w1}. For any $t\geq 0$, it holds 
\begin{align}
    \EE_t \normtxt{\hat{\prm}_{t+1}}^2
    &\leq (1 \!-\! 2\mu \gamma_{t+1}) \normtxt{\hat{\prm}_t}^2 \!+\! \gamma_{t+1}^2 \big( \min\{c^2, G^2 \} + d\sigmaDP{2}\big) 
     - 2\gamma_{t+1}\Pscal{\hat{\prm}_{t}}{\EE_t[b_t]} 
    .\label{lemeq:onedes-pcsgd}
\end{align}
\end{lemma}
The proof is in \S\ref{app:onedes-pcsgd}. 
The inner product term above involving $\EE_t[b_t]$ captures the clipping bias and the distribution shift. The latter is the difference between the clipped stochastic gradient and the expected gradient induced by ${\cal D}(\prm_{PS})$.

Such term is unlikely to be small except for a large clipping threshold $c$. For example, \citep[Lemma 9]{zhang2020adaptive} applied an indicator function trick to bound $\| \EE_t [ b_t ] \|$ by $G / c$. We improve their treatment on the bias term via a better use of the smoothness property (cf.~Lemma~\ref{lem:continuity}) to obtain:
\begin{lemma}\label{lem:inner}
Under A\ref{assu:lips}, \ref{assu:bndgrd}, \ref{assu:w1},  the following upper bound holds:
\begin{align}
    &- 2\gamma_{t+1}\Pscal{\prm_t - \prm_{PS}}{ \EE_t [b_t] }\leq 2L \beta \gamma_{t+1}\norm{\prm_{t}-\prm_{PS}}^2  + \tmu \gamma_{t+1}\norm{\prm_t-\prm_{PS}}^2  + \frac{4 \gamma_{t+1}}{\tmu}\cdot  {\cal C}_1, \label{lemeq:inner}
\end{align}
with the constants ${\cal C}_1 \eqdef (\max\{G-c, 0\})^2$, $\tmu \eqdef \mu - L \beta$.
\end{lemma}
See \S \ref{app:inner-pcsgd} for the detailed proof.
The first two terms on the right hand side of Eq.~\eqref{lemeq:inner} vanishes as $\| \prm_t - \prm_{PS} \| \to 0$. Meanwhile, the term $\frac{4 \gamma_{t+1} {\cal C}_1}{\tmu}$ has led to an inevitable clipping bias. Notice that this term vanishes only in the trivial case of $c \geq G$, i.e., the clipping threshold is larger than any stochastic gradient. Otherwise, this bias will propagate through the algorithm and lead to an asymptotic bias.

\begin{proof}[\bf Proof of Theorem \ref{thm1}]
Combining Lemmas \ref{lem:onedes}, \ref{lem:inner} leads to the following recursion:
\begin{align}
    \EE_t \normtxt{\hat{\prm}_{t+1} }^2 &\leq   (1- \tmu \gamma_{t+1}) \normtxt{\hat{\prm}_t}^2 + \frac{4 {\cal C}_1}{\tmu} \gamma_{t+1} + \gamma_{t+1}^2 \left( 2c^2 + 2 G^2 +\! d\sigmaDP{2} \right). 
    \label{eq:dd}
\end{align}
Recall that  $c_1\eqdef (2 c^2 + 2 G^2 + d\sigmaDP{2})$.
Taking full expectation on both sides of \eqref{eq:dd} leads to:
\begin{align}\label{eq:contraction_proof_thm1}
    \EE \normtxt{\hat{\prm}_{t+1}}^2 & \leq  \prod_{i=1}^{t+1}(1- \tmu \gamma_{i}) \norm{\hat{\prm}_0}^2 + c_1 \sum_{i=1}^{t+1} \gamma_{i}^2 \prod_{j=i+1}^{t+1} (1- \tmu \gamma_{j} +\frac{4 {\cal C}_1}{\tmu } \sum_{i=1}^{t+1} \gamma_{i} \prod_{j=i+1}^{t+1} (1- \tmu \gamma_{j})  \notag
    \\
    &  \leq \prod_{i=1}^{t+1}(1- \tmu \gamma_{i}) \norm{\hat{\prm}_0 }^2 + \frac{8 {\cal C}_1}{\tmu^2} + \frac{2 c_1}{\tmu} \gamma_{t+1},
\end{align}
where we used Lemma \ref{lem:aux} in the last inequality for step size satisfying $\sup_{t\geq 1} \gamma_{t} \leq \frac{2}{\Tilde{\mu}}$. 
\end{proof}\vspace{-.1cm}

\subsection{Proof of Lemma \ref{lem:onedes}}\label{app:onedes-pcsgd}
\begin{proof}
Recall the notation $\Tgrd g(\prm_{t}) \eqdef \clip_{c} (\grd \ell(\prm_{t}; Z_{t+1}))$ which we introduce in (\ref{eq:convention}). Then, we deduce the following chain
\begin{align*}
     \norm{\prm_{t+1} - \prm_{PS}}^2 &\overset{(a)}{=} \norm{\proj_{\cal X} \left( \prm_t - \gamma_{t+1}\left[ \Tgrd g(\prm_{t}) + \zeta_{t+1}\right]\right) -\proj_{\cal X}\left( \prm_{PS} + \gamma_{t+1} \grd f( \prm_{PS} ; \prm_{PS} ) \right)}^2
     \\
     &\overset{(b)}{\leq}  \norm{\prm_t - \gamma_{t+1}\left[ \Tgrd g(\prm_{t}) + \zeta_{t+1}\right] - \prm_{PS} + \gamma_{t+1} \grd f( \prm_{PS} ; \prm_{PS} ) }^2
    \\
    &= \norm{\prm_t - \prm_{PS}}^2 + \gamma_{t+1}^2 \norm{\Tgrd g(\prm_t) + \zeta_{t+1} - \grd f( \prm_{PS} ; \prm_{PS} ) }^2 \\
    & \quad - 2\gamma_{t+1}\Pscal{\prm_t - \prm_{PS}}{\Tgrd g(\prm_{t}) + \zeta_{t+1} - \grd f( \prm_{PS} ; \prm_{PS} ) },
\end{align*}
where in equality $(a)$, we applied the definition of $\prm_{PS}$ in (\ref{eq:convention}). Inequality $(b)$ is due to the  non-expansive property of the projection operator. Introducing notation $b_t \eqdef \Tgrd g(\prm_t) - \grd f(\prm_t; \prm_{PS})$ which we define in \eqref{eq:convention} into the above inequality gives us
\begin{align*}
    \norm{\prm_{t+1} - \prm_{PS}}^2 &\leq \norm{\prm_t - \prm_{PS}}^2 + \gamma_{t+1}^2 \norm{\Tgrd g(\prm_t) + \zeta_{t+1} - \grd f( \prm_{PS} ; \prm_{PS} ) }^2 \\
    &\quad - 2\gamma_{t+1}\Pscal{\prm_t - \prm_{PS}}{b_t + \grd f(\prm_t; \prm_{PS}) - \grd f( \prm_{PS} ; \prm_{PS} ) }
    \\
    &\overset{(a)}{\leq} (1 - 2 \gamma_{t+1} \mu ) \norm{\prm_t - \prm_{PS}}^2 + \gamma_{t+1}^2 \norm{\Tgrd g(\prm_t) + \zeta_{t+1} - \grd f( \prm_{PS} ; \prm_{PS} ) }^2 
    \\
    &\quad - 2\gamma_{t+1}\Pscal{\prm_t - \prm_{PS}}{b_t},
\end{align*}
where inequality $(a)$ is due to strong convexity of $\ell(\cdot; z)$, i.e., $\Pscal{\prm_t - \prm_{PS}}{ \grd f( \prm_t; \prm_{PS} ) - \grd f(\prm_{PS};\prm_{PS})} \geq \mu \| \prm_t - \prm_{PS} \|^2$.
 Taking conditional expectation with regards to $\prm_t$ on both sides gives us
\beq\label{eq:bb}
\begin{aligned}
    \EE_t \norm{\prm_{t+1} - \prm_{PS}}^2 &\leq(1- 2\mu \gamma_{t+1}) \norm{\prm_t - \prm_{PS}}^2 - 2\gamma_{t+1}\Pscal{\prm_t - \prm_{PS}}{ \EE_t [b_t] } 
    \\
    &\quad + \gamma_{t+1}^2 \EE_t \norm{\Tgrd g(\prm_t) + \zeta_{t+1} - \grd f(\prm_{PS}; \prm_{PS})}^2. 
\end{aligned}
\eeq
For the last term of right-hand-side of the above inequality, we have
\begin{align*}
    &\EE_{t} \norm{\Tgrd g(\prm_t) + \zeta_{t+1} - \grd f(\prm_{PS}; \prm_{PS})}^2
    \\
    &= \EE_{t} \left[ \normtxt{\Tgrd g(\prm_t) - \grd f(\prm_{PS}; \prm_{PS})}^2 + \pscal{\Tgrd g(\prm_t) - \grd f(\prm_{PS}; \prm_{PS})}{\zeta_{t+1}} + \norm{\zeta_{t+1}}^2 \right]
    \\
    &\stackrel{(a)}{=}  \EE_{t} \normtxt{\Tgrd g (\prm_t) - \grd f(\prm_{PS}; \prm_{PS})}^2 + 0 + {d}\sigmaDP{2} 
    \\
    &= \EE_t \normtxt{\clip_{c}(\grd \ell(\prm_t; z_{t+1})) - \grd f(\prm_{PS}; \prm_{PS})}^2 + d\sigmaDP{2}
    \\
    &\leq 2\EE_t \norm{\clip_c(\grd \ell (\prm_t; z_{t+1}))}^2 + 2\norm{\grd f(\prm_{PS}; \prm_{PS})}^2 + d\sigmaDP{2}
    \overset{(b)}{\leq} 2(c^2 + G^2) + d\sigmaDP{2},
\end{align*}
where $(a)$ is obtained since the additive perturbing noise is zero mean and statistically independent of the stochastic gradient, $(b)$ is due to A~\ref{assu:bndgrd} and definition of clipping operator.

\vspace{+.2cm}
\noindent
Substituting above upper bound to (\ref{eq:bb}) leads to
\begin{align*}
    \EE_t \norm{\prm_{t+1} - \prm_{PS}}^2 &\leq  \left(1- 2\mu  \gamma_{t+1} \right) \norm{\prm_t - \prm_{PS}}^2 - 2\gamma_{t+1}\Pscal{\prm_t - \prm_{PS}}{ \EE_t [b_t] }  
    \\
    &\quad + \gamma_{t+1}^2 \left( 2c^2 + 2G^2 + d \sigmaDP{2} \right),
\end{align*}
which finishes the proof of Lemma \ref{lem:onedes}.
\end{proof}

\subsection{Proof of Lemma \ref{lem:inner}}\label{app:inner-pcsgd}
\begin{proof}

Applying the Cauchy-Schwarz inequality on the inner product $\Pscal{ \prm_{PS}-\prm_t}{ \EE_t [b_t] }$ yields the following upper bound:
\[
- 2\gamma_{t+1}\Pscal{\prm_t - \prm_{PS}}{ \EE_t [b_t] } \leq  2\gamma_{t+1}\norm{\prm_t - \prm_{PS}}\cdot \norm{ \EE_t [b_t] }.
\]
We now proceed to upper bounding  $\norm{ \EE_t [b_t] }$:
\begin{align*}
    \norm{\EE_t [b_{t}]}  &= \norm{\EE_t \left[ \Tgrd g(\prm_t) - \grd f(\prm_t; \prm_{PS}) \right] } \\
    &= \norm{\grd g(\prm_t;  \prm_t) - \grd f(\prm_t; \prm_{PS})}
    \\
    &\leq \norm{\grd g(\prm_t;  \prm_t) - \grd f(\prm_t; \prm_{t})} + \norm{ \grd f(\prm_t; \prm_{t}) - \grd f(\prm_t;  \prm_{PS})}
    \\ 
    & \leq \norm{\EE_{z\sim {\cal D}(\prm_t)} \left( \clip_{c}(\grd \ell(\prm_t; z)) - \grd \ell(\prm_t; z) \right)} + L\beta \norm{\prm_t - \prm_{PS}}
\end{align*}
where the last inequality is due to Assumption~\ref{assu:lips} and Lemma D.4 of \citep{perdomo2020performative}. Moreover,
\beq \label{eq:lb_clip_ell}
\begin{aligned}
    & \norm{\EE_{z\sim {\cal D}(\prm_t)}\left[ \min \left(1, \frac{c}{\norm{\grd \ell(\prm_t; z)}}\right) - 1\right]\grd \ell(\prm_t; z)} 
    \\
    & \leq \EE_{z\sim {\cal D}(\prm_t)} \left[ \left| 1 - \min\left( 1, \frac{c}{\norm{\grd \ell(\prm_t; z )}}\right)\right| \cdot \norm{\grd \ell(\prm_t; z)}\right] 
    \\
    & = \EE_{z\sim {\cal D}(\prm_t)} \left[ \max\left( 0, {\norm{\grd \ell(\prm_t; z )} -c }\right) \right] \leq \max\{ G-c, 0\} ,
\end{aligned}
\eeq
where we have used Assumption~\ref{assu:bndgrd} in the last inequality.
Finally, we obtain
\begin{align*}
    - 2\gamma_{t+1}\Pscal{\prm_t - \prm_{PS}}{ \EE_t [b_t] } &\leq   2\gamma_{t+1} \norm{\prm_t - \prm_{PS}} \cdot \left( \max\{ G-c, 0\} + L\beta \norm{\prm_t - \prm_{PS}} \right)
    \\
    &= 2L\beta \gamma_{t+1}\norm{\prm_{t}-\prm_{PS}}^2 + 2\gamma_{t}\norm{\prm_t-\prm_{PS}} \cdot \max\{ G-c, 0\}
    \\
    &\leq 2L\beta\gamma_{t+1}\norm{\prm_{t}-\prm_{PS}}^2 + \tmu \gamma_{t+1}\norm{\prm_t-\prm_{PS}}^2 + \frac{4 \gamma_{t+1}}{\tmu}\cdot (\max\{ G-c, 0\})^2,
\end{align*}
where in the last inequality, we used the H\"older's inequality $xy \leq {ax^2} + \frac{y^2}{a}$ and set $a=\tmu/2.$
\end{proof}

\section{Proof of Theorem \ref{thm:lb}}\label{app:lb}

\begin{proof}
    Let $a,b,\beta > 0$ and denote by ${\cal B}(p)$ the Bernoulli distribution with mean $0<p<\frac{1}{2}.$
    We consider the following quadratic loss function:
    \beq\label{eq:quad}
        &\min_{\prm\in {\cal X}} \EE_{Z \sim {\cal D}(\prm)} [\ell(\prm;Z)],\nonumber\\
        &~~\text{ where}~\ell(\prm;z) = \frac{1}{2} (\prm + az)^2~\text{and}~Z \sim {\cal D}(\prm) \Longleftrightarrow Z = b \tilde{Z} - \beta \prm,~\tilde{Z} \sim {\cal B}(p).
    \eeq
   We require that 
   $ab \geq 2c$. When $Z \sim {\cal D}(\prm)$, we note that the stochastic gradient is in the form:
    \[
        \grd \ell(\prm; Z) = \prm + a( b \tilde{Z} - \beta \prm ) = (1-a \beta) \prm + ab  \tilde{Z}.
    \]
    Problem~\eqref{eq:quad} satisfies {A\ref{assu:lips} with $\mu = 1$, $L=a$ and A\ref{assu:w1} with the parameter $\beta$}. 
    
    Consider the following point:
    \[ 
        \prminf = -\frac{pc}{(1-p)(1-a\beta)}.
    \]
    We claim that $\prminf$ is a fixed point of the clipped SGD algorithm. 
    Note that the stochastic gradient at $\prminf$ with $Z \sim {\cal D}( \prminf )$ when $\tilde{Z}=0$ is not clipped, since
    \[
        | (1- a\beta)\prminf + ab \cdot 0 |  = \left| (1-a\beta)\times \left( -\frac{pc}{(1-p)(1-a\beta)} \right) \right| = \left| - \frac{pc}{1-p} \right| \leq c,
    \]
    where the last inequality is due to the assumption $p<\frac{1}{2}$. Meanwhile, when $\tilde{Z} = 1$, the stochastic gradient is:
    \[
        (1-a\beta) \prminf + ab = -\frac{pc}{1-p} + ab \geq - c + 2c = c,
    \]
    where the last inequality uses the condition $ab \geq 2c$. In particular, we have that:
    \beqq
        \EE_{Z\sim {\cal D}(\prm)} [\clip_c(\grd \ell(\prminf; Z))] = (1-p)\cdot [(1-a\beta)\prm] + pc = 0 .
    \eeqq
    The above shows that $\prminf$ is a fixed point of the Clipped SGD algorithm. 

    On the other hand, the performative stable solution, $\prm_{PS}$, of problem (\ref{eq:quad}) solves the following equation, 
    \beqq
     \EE_{z\sim {\cal D}(\prm_{PS})} [\grd \ell(\prm_{PS}; z)] = \EE_{Z\sim {\cal D}(\prm_{PS})} (\prm_{PS} + a Z) = \prm_{PS} + a (p - \beta\prm_{PS}) = 0,
    \eeqq
    we get that 
    \[
        \prm_{PS} = -\frac{pa}{1-a\beta}.
    \]
    Finally, we can lower bound the asymptotic bias of clipped SGD as
    \begin{align*}
    \norm{\prminf - \prm_{PS}}^2 &= \norm{-\frac{pc}{(1-p)(1-a\beta)} + \frac{pa}{1-a\beta}}^2 = \frac{p^2}{(1-a\beta)^2} \left( a-\frac{c}{1-p} \right)^2 
    = \Omega\left(\frac{1}{(\mu - L\beta)^2}\right),
    \end{align*}
    
    Finally, we aim to prove that $\prminf$ is unique. We remark that the limiting points of clipping SGD satisfies
    \beqq 
        \EE_{z\sim {\cal D}(\prm)} \clip_c (\grd \ell(\prm; z)) = 0,
    \eeqq
    which is equivalent to 
    \beq\label{eq:limiting}
        p \times \clip_c[(1-a\beta)\prm + ab] + (1-p) \times \clip_c[(1-a\beta)\prm] = 0.
    \eeq
    {\bf Case 1}: $(1-a\beta)\prm + ab > c$ and $(1-a\beta)\prm < c$, then, we will get $\prminf$.

    {\bf Case 2}: $(1-a\beta)\prm + ab > c$ and $(1-a\beta)\prm > c$, then
    \[
        p\times c + (1-p) \times c =1 \neq 0.
    \]

    {\bf Case 3}: $(1-a\beta)\prm + ab < c$ and $(1-a\beta)\prm < c$, i.e, it requires that $\prm < \frac{c-ab}{1-a\beta}$, then from (\ref{eq:limiting}), we obtain that 
    \begin{align*}
        & p \times [(1-a\beta)\prm + ab] + (1-p) \times [(1-a\beta)\prm] = 0.
    \end{align*}
    Solve it, we get $\prm  = \frac{pab}{1-a\beta}$. Note that 
    \[
        \prm = \frac{pab}{1-a\beta} \leq \frac{c-ab}{1-a\beta},
    \]
    which is impossible, since $ab\geq 2c$ and $p, a>0$.
    In conclusion, there is only one solution of (\ref{eq:limiting}).
    This concludes the proof.
\end{proof}

\section{Proof of Theorem~\ref{thm:nvcx}} \label{app:noncvx}

In this section, we consider a special case of \eqref{eq:p} where ${\cal X} = \RR^d$. Similar to \S\ref{app:pf_scvx}, we also consider the general setting of {\pcsgd} with DP noise \eqref{algo:pcsgd-dp}.

\begin{theorem*}
    Under A\ref{assu:lips}, \ref{assu:bndgrd}, \ref{assu:var-ncvx}, \ref{assu:tv}, \ref{assu:bd_loss}. Let the step sizes satisfy $\sup_{t\geq 1} \gamma_{t} \leq \frac{1}{2(1+\sigma_1^2)}$. Then, for any $T\geq 1$, the iterates $\{ \prm_{t} \}_{t\geq 0}$ generates by \eqref{algo:pcsgd} satisfy:
    \begin{align}
    \sum_{t=0}^{T-1} \gamma_{t+1} \EE \left[ \norm{\grd f(\prm_t; \prm_t)}^2 \right] \leq 8 \Delta_0 + 4L (\sigma_0^2+\sigmaDP{2}) \sum_{t=0}^{T-1}\gamma_{t+1}^2 + 8 {\sf b}(\beta,c) \sum_{t=0}^{T-1}\gamma_{t+1}, \notag
    \end{align}
    where $\Delta_0 \eqdef \EE[ f(\prm_0; \prm_0) - \ell^\star]$ is an upper bound to the initial optimality gap for performative risk, and
    \beq\notag
    {\sf b}(\beta,c) \eqdef \widehat{L} \beta \left( \sqrt{\sigma_0^2+\sigmaDP{2}} + 8(1+\sigma_1^2)\widehat{L}\beta \right) + 2\max\{G-c, 0\}^2. 
    \eeq
\end{theorem*}

\begin{proof}[Proof of Theorem~\ref{thm:nvcx} with general $\sigmaDP{2}$]
For fixed $z\in {\sf Z}$, applying A\ref{assu:lips} leads to 
\begin{align*}
    \ell(\prm_{t+1}; z) &\leq \ell(\prm_t; z) + \pscal{\grd \ell(\prm_t; z)}{\prm_{t+1} - \prm_{t}} + \frac{L}{2}\norm{\prm_{t+1} - \prm_{t}}^2
    \\
    &= \ell(\prm_t; z) - \gamma_{t+1} \pscal{\grd \ell(\prm_t; z)}{\clip_{c}\left( \grd \ell(\prm_t; z_{t+1})\right) + \zeta_{t+1}} + \frac{L \gamma_{t+1}^2}{2}\norm{\clip_{c}(\grd \ell(\prm_{t; z_{t+1}})) + \zeta_{t+1}}^2
    \\
    &= \ell(\prm_t; z) - \gamma_{t+1} \pscal{\grd \ell(\prm_t; z)}{ \Tgrd g(\prm_t) + \zeta_{t+1}} + \frac{L \gamma_{t+1}^2}{2}\norm{ \Tgrd g(\prm_t) + \zeta_{t+1}}^2.
\end{align*}
The second line is due to the update rule of \eqref{algo:pcsgd}. In the last line, we recall the notation $\Tgrd g(\prm_t) \eqdef \clip_{c}(\grd \ell(\prm_t; z_{t+1}))$. Taking integration on fixed $z$ with weights given by the p.d.f. of ${\cal D}(\prm_t)$, i.e., $\int(\cdot) p_{\prm_t}(z) \diff z$ on both sides of above inequality yield
\begin{align*}
    f(\prm_{t+1}; \prm_{t}) \leq f(\prm_{t}; \prm_{t}) - \gamma_{t+1} \pscal{\grd f(\prm_t; \prm_{t})}{\Tgrd g(\prm_{t}) + \zeta_{t+1}} + \frac{L \gamma_{t+1}^2}{2} \norm{\Tgrd g(\prm_{t}) + \zeta_{t+1}}^2.
\end{align*}
Taking conditional expectation $\EE_{t}[\cdot]$ on the clipping stochastic gradients leads to
\begin{align}\label{eq:ncvx-1}
    f(\prm_{t+1}; \prm_{t}) \leq f(\prm_{t}; \prm_{t}) - \gamma_{t+1} \pscal{\grd f(\prm_t; \prm_{t})}{\EE_{t} \left[\Tgrd g(\prm_{t})\right] } + \frac{L \gamma_{t+1}^2}{2} \EE_{t}\left[\norm{\Tgrd g(\prm_{t}) + \zeta_{t+1}}^2 \right].
\end{align}
where we use the property that $\zeta_{t+1} \sim {\cal N}(0, \sigmaDP{2}{\bm I})$. For the last term in above, we observe the following chain,
\begin{align*}
    \EE_{t}\left[\norm{\Tgrd g(\prm_{t}) + \zeta_{t+1}}^2 \right] &= \EE_{t} \norm{\Tgrd g(\prm_t)}^2 + 0 + \sigmaDP{2} \overset{(a)}{\leq} \EE_{t} \norm{\grd \ell(\prm_t; z_{t+1})}^2 + \sigmaDP{2}
    \\
    &= \EE_{t} \norm{\grd \ell(\prm_{t}; z_{t+1}) - \grd f(\prm_t; \prm_t)}^2 + \norm{\grd f(\prm_t; \prm_t)}^2 + \sigmaDP{2}
    \\
    &\overset{(b)}{\leq} \sigma_0^2 + \sigmaDP{2} + (1+\sigma_1^2) \norm{\grd f(\prm_t; \prm_t)}^2, 
\end{align*}
where $(a)$ used the definition of clipping operator \eqref{eq:clip_op} and $(b)$ used A\ref{assu:var-ncvx}. Substituting above results to \eqref{eq:ncvx-1} gives us
\begin{align}\label{eq:ncvx-2}
    f(\prm_{t+1}; \prm_{t}) \leq f(\prm_{t}; \prm_{t}) - \gamma_{t+1} \Pscal{\grd f(\prm_t; \prm_{t})}{\EE_{t} \left[\Tgrd g(\prm_{t})\right] } + \frac{L \gamma_{t+1}^2}{2} \left( \sigma_0^2 + \sigmaDP{2} + (1+\sigma_1^2) \norm{\grd f(\prm_t; \prm_t)}^2 \right).
\end{align}
The inner product term can be lower bounded by
\begin{align*}
    \Pscal{\grd f(\prm_t; \prm_{t})}{\EE_{t} \left[\Tgrd g(\prm_{t})\right] } &= \norm{\grd f(\prm_t; \prm_t)}^2 - \pscal{\grd f(\prm_t; \prm_t)}{\grd g(\prm_t) - \grd f(\prm_t; \prm_t)}
    \\
    &\geq \norm{\grd f(\prm_t; \prm_t)}^2 - \norm{\grd f(\prm_t; \prm_t)}\cdot \norm{\grd g(\prm_t) - \grd f(\prm_t; \prm_t)}
    \\
    &\overset{(a)}{\geq} \norm{\grd f(\prm_t; \prm_t)}^2 - \max\{G-c, 0\} \norm{\grd f(\prm_t; \prm_t)}
    \\
    &\geq \frac{1}{2}\norm{\grd f(\prm_t; \prm_t)}^2 - 2\max\{G-c, 0\}^2 ,
\end{align*}
where $(a)$ is due to \eqref{eq:lb_clip_ell} in Lemma \ref{lem:inner} and the last inequality is due to the fact that $xy\leq ax^2 + \frac{y^2}{a}$, for any $a>0.$ Substituting back to \eqref{eq:ncvx-2} gives
\begin{align*}
    f(\prm_{t+1}; \prm_{t}) \leq f(\prm_{t}; \prm_{t}) - \frac{1}{2}\norm{\grd f(\prm_t; \prm_t)}^2 + 2\max\{G-c, 0\}^2  + \frac{L \gamma_{t+1}^2}{2} \left( \sigma_0^2 + \sigmaDP{2} + (1+\sigma_1^2) \norm{\grd f(\prm_t; \prm_t)}^2 \right).
\end{align*}
Rearranging terms,
\begin{align*}
    \left(\frac{1}{2} - \frac{L(1+\sigma_1^2)\gamma_{t+1}}{2} \right) \norm{\grd f(\prm_t; \prm_t)}^2 &\leq f(\prm_{t}; \prm_{t}) - f(\prm_{t+1}; \prm_{t}) + 2\max\{G-c, 0\}^2  + \frac{L (\sigma_0^2 + \sigmaDP{2} )}{2} \gamma_{t+1}^2.
\end{align*}
The step size condition $\sup_{t\geq 1} \gamma_{t} \leq \frac{1}{2(1+\sigma_1^2)}$ implies that $\frac{1}{2} - \frac{L(1+\sigma_1^2)\gamma_{t+1}}{2} \geq \frac{1}{4}$.
\begin{align}\label{eq:ncvx-main}
    \frac{1}{4}\norm{\grd f(\prm_t; \prm_t)}^2 &\leq f(\prm_{t}; \prm_{t}) - f(\prm_{t+1}; \prm_{t}) + 2\max\{G-c, 0\}^2  + \frac{L (\sigma_0^2 + \sigmaDP{2} )}{2} \gamma_{t+1}^2.
\end{align}
Next, taking full expectation on both sides and decomposing the term $f(\prm_{t}; \prm_{t}) - f(\prm_{t+1}; \prm_{t})$ as:

\begin{align}\label{eq:ft-ft1}
    \EE [f(\prm_{t}; \prm_{t}) - f(\prm_{t+1}; \prm_{t})] &= \EE[f(\prm_{t}; \prm_{t}) - f(\prm_{t+1}; \prm_{t+1})] + \EE[f(\prm_{t+1}; \prm_{t+1}) - f(\prm_{t+1}; \prm_{t})]
    \\
    &\leq \EE[f(\prm_{t}; \prm_{t}) - f(\prm_{t+1}; \prm_{t+1})] + \widehat{L} \beta \norm{\prm_{t+1} - \prm_{t}} \notag
    \\
    &\overset{(a)}{\leq} \EE[f(\prm_{t}; \prm_{t}) - f(\prm_{t+1}; \prm_{t+1})] + \widehat{L} \beta \gamma_{t+1} \EE_{t} \norm{\Tgrd g(\prm_t) + \zeta_{t+1} } \notag
    \\
    &\overset{(b)}{\leq} \EE[f(\prm_{t}; \prm_{t}) - f(\prm_{t+1}; \prm_{t+1})] + \widehat{L} \beta \gamma_{t+1} \sqrt{\EE_{t} \norm{\Tgrd g(\prm_t) + \zeta_{t+1}}^2 } \notag
    \\
    &\overset{(c)}{\leq} \EE[f(\prm_{t}; \prm_{t}) - f(\prm_{t+1}; \prm_{t+1})] + \widehat{L} \beta \gamma_{t+1} \left( \sqrt{ \sigma_0^2 + \sigmaDP{2}} + \sqrt{(\sigma_1^2 + 1) } \norm{\grd f(\prm_t; \prm_t)} \right) \notag
    \\
    &\leq \EE[f(\prm_{t}; \prm_{t}) - f(\prm_{t+1}; \prm_{t+1})] \notag
    \\
    &\quad + \widehat{L} \beta \gamma_{t+1} \left( \sqrt{ \sigma_0^2 + \sigmaDP{2}} + (1+\sigma_1^2)\cdot 8 \widehat{L}\beta + \frac{1}{8\widehat{L}\beta}\norm{\grd f(\prm_t; \prm_t)}^2 \right). \notag
\end{align}
Inequality $(a)$ is due to the update rule \eqref{algo:pcsgd} and $(b)$ is implied by $\left[\EE{X} \right]^2 \leq \EE[X^2]$. In $(c)$, we apply A\ref{assu:var-ncvx}. Back to \eqref{eq:ncvx-main}, we have
\begin{align*}
    \frac{1}{8} \gamma_{t+1} \EE \norm{\grd f(\prm_t; \prm_t)}^2 &\leq \EE\left[f(\prm_t; \prm_t) - f(\prm_{t+1}; \prm_{t+1}) \right] + \frac{L (\sigma_0^2 + \sigmaDP{2}) }{2} \gamma_{t+1}^2 + {\sf b}(\beta,c) \cdot \gamma_{t+1},
\end{align*}
where ${\sf b}(\beta, c)$ is a constant depending on distribution shifting strength and clipping threshold, which defined as
\beq\label{def:b(eps)}
    {\sf b}(\beta,c) \eqdef \widehat{L}\beta \left( \sqrt{ \sigma_0^2 + \sigmaDP{2}} + 8(1+\sigma_1^2)\widehat{L}\beta \right) + 2\max\{G-c, 0\}^2. 
\eeq
Taking summation from $t=0,1,\cdots, T-1$ gives us
\begin{align*}
    \frac{1}{8} \sum_{t=0}^{T-1}\gamma_{t+1} \EE \norm{\grd f(\prm_t; \prm_t)}^2 &\leq \Delta_0 + \frac{L (\sigma_0^2 + \sigmaDP{2}) }{2} \sum_{t=0}^{T-1}\gamma_{t+1}^2 + {\sf b}(\beta,c) \cdot \sum_{t=0}^{T-1}\gamma_{t+1}.
\end{align*}
Set $\gamma_{t} = 1/\sqrt{T}$ and divide $\sum_{t=0}^{T-1} \gamma_{t+1}$ on both sides yields the theorem.
\end{proof}

\section{Proof of Corollary \ref{thm:dp}} \label{app:privacy}
\begin{proof}

Let $o$ and $aux$ denote an outcome and an auxiliary input, respectively. 
We wish to prove 
\[
\Pr({\cal M}(aux, {\bm Z}) = o)  \leq e^{\varepsilon} \Pr({\cal M}(aux, {\bm Z}^\prime) = o) + \delta,
\]
where ${\bm Z}$ and ${\bm Z^\prime}$ are two adjacent datasets, i.e., they are only different by one sample.
Let us define the privacy loss of an outcome $o$ on two datasets ${\bm Z}$ and ${\bm Z}^\prime$ as 
\[
    c(o; {\cal M}, aux, {\bm Z}, {\bm Z}^\prime) \eqdef \log \frac{\Pr({\cal M}(aux, {\bm Z}) = o)}{\Pr({\cal M}(aux, {\bm Z}^\prime) = o)}
\]
and its log moment generating function as 
\[
    \alpha^{\cal M}(\lambda; aux, {\bm Z}, {\bm Z}^\prime) \eqdef \log \EE_{o \sim {\cal M}(aux, {\bm Z})} [\exp(\lambda \cdot c(o; {\cal M}, aux, {\bm Z}, {\bm Z}^\prime) )],
\]
where ${\cal M}$ is the mechanism we focus on, $\lambda >0$ is the variable of log moment generating function. Taking maximum over conditions, the unconditioned log moment generating function is 
\[
    \hat{\alpha}^{\cal M}(\lambda) \eqdef \max_{aux, {\bm Z}, {\bm Z}^\prime} \alpha^{\cal M}(\lambda; aux, {\bm Z}, {\bm Z}^\prime).
\]
The overall log moment generating function can be bounded as following according to composability in \citep[Theorem 2]{abadi2016deep}
\[
    \hat{\alpha}^{\cal M}(\lambda) \leq \sum_{t=1}^{T} \hat{\alpha}^{{\cal M}^{(t)}}(\lambda).
\]
Let 
{$q=\frac{1}{m}$} denotes the probability each data sample is drawn {uniformly} from the dataset ${\sf D}_0$. 
{Recall \eqref{eq:shift_data}, which leads to $\bar{z}_i = Z-s_i(\theta)$ where $i$ is chosen uniformly from $[m]$ and $s_i(\theta)$ is generated according to the random mapping $\mathcal{S}_i: {\cal X} \rightarrow {\sf Z}$. We can define this relationship by the random mapping $f: ({\sf Z}+{\sf D}_0)\rightarrow {\sf D}_0$, where for $X\subseteq \mathbb{R}^d$ and $Y\subseteq \mathbb{R}^d$ we let $X+Y$ denote the set with all the variables $\{x+y:x\in X,\: y\in Y\}$ and the randomness follows from the uniform distribution of $i\in[m]$ and the random mapping $\mathcal{S}_i: {\cal X} \rightarrow {\sf Z}$. By \citep[Proposition 2.1]{DF_algoithmic_foundation} using the post processing function $f$, after the privacy mechanism $\mathcal{M}$, i.e. $f\circ \mathcal{M}$,  does not compromise  differential privacy.}

{Observe that the distribution shift due to ${\cal S}_i(\prm)$ in \eqref{eq:shift_data} cannot compromise differential privacy \citep[Proposition 2.1]{DF_algoithmic_foundation}.} Let $q \leq \frac{c}{16 \sigmaDP{}}$ and $\lambda \leq \frac{\sigmaDP{2}}{c^2} \log \frac{c}{ q  \sigmaDP{}}$, we apply \citep[Lemma 3]{abadi2016deep} to bound the unconditioned log moment generating function as 
\[
 \hat{\alpha}^{{\cal M}^{(t)}}(\lambda) \leq \frac{q^2 \lambda(1+\lambda) c^2}{(1-q)\sigmaDP{2}} + {\cal O}\left( \frac{q^3 \lambda^3c^3 }{\sigmaDP{3}}\right) = {\cal O}\left( \frac{q^2 \lambda^2 c^2}{\sigmaDP{2}}\right).
\]

Finally, we only need to verify that there exists some $\lambda$ that satisfies the following inequalities
\begin{align*}
    & {T} \left(\frac{q c \lambda}{\sigmaDP{}}\right)^2 \leq \frac{\lambda \varepsilon}{2}, \quad 
    \exp(-\lambda \varepsilon /2 ) \leq \delta,
    \quad 
     \lambda \leq \frac{\sigmaDP{2}}{c^2} \log \left(\frac{c}{q \sigmaDP{}}\right).
\end{align*}
We can verify that when $\varepsilon=c_1 q^2 T$, $q = 1/m$ and $\sigmaDP{} = \frac{cq\sqrt{T \log(1/\delta)}}{\varepsilon}$, all above conditions can be satisfied for some explicit constant $c_1$. This finishes the proof.
\end{proof}

\section{Proof of Theorem \ref{thm:dicesgd}} \label{app:dice}

For the convenience of derivations, we introduce the following notations: $\Tprm_t = \prm_t - \gamma_t e_t$. Its update rule is 
\beq \label{eq:update}
    \Tprm_{t+1} &=& \prm_{t+1} - \gamma_{t+1} e_{t+1} 
     \notag 
    \\
    &=& \prm_t - \gamma_{t+1}(v_{t+1} + \zeta_t) - \gamma_{t} \left( e_t + \grd \ell(\prm_t; z_t) - v_{t+1} \right)
    \\
    &=& \prm_t - \gamma_{t+1}\zeta_{t+1} - \gamma_{t+1} \left( e_t + \grd \ell(\prm_t; z_{t+1})\right)
    \notag 
    \\
    &=& \Tprm_t + \left( \gamma_{t} - \gamma_{t+1}\right) e_t - \gamma_{t+1} \left( \zeta_{t+1}  + \grd \ell(\prm_t; z_{t+1}) \right).
    \notag
\eeq
We remark that if we choose a constant step size, i.e, $\gamma_{t} \equiv \gamma$, then $\Tprm_{t+1} = \Tprm_t - \gamma \left( \zeta_{t+1}  + \grd \ell(\prm_t; z_{t+1}) \right)$, which degenerates to the case in the proof of \cite{zhang2023differentially}.

Meanwhile, we can obtain the feedback error $e_t$'s update rule:
\beq\label{eq:er_update}
    e_{t+1} &=& e_t + \grd \ell(\prm_t; z_{t+1}) - \clip_{C_2}(e_t) - \clip_{C_1}(\grd \ell(\prm_t; z_{t+1}))
    \\
    &=& (1-\alpha_e^t) e_{t} + (1-\alpha^t) \grd \ell(\prm_t; z_{t+1}), \notag
\eeq
where we define $\alpha_e^t$, $\alpha^t$ as following:
\beq\label{def:alpha_et&alpha_t}
   \fet{t} \eqdef \min\left\{1, \frac{C_1}{\norm{e_t}} \right\}, \quad \alpha^t \eqdef \min\left\{ 1, \frac{C_2}{\norm{\grd \ell(\prm_t; z_{t+1})}}\right\}.
\eeq
We aim to analyze the squared distance between $\Tprm_{t+1}$ and $\prm_{PS}$. 
\begin{align*}
    \norm{\Tprm_{t+1} - \prm_{PS}}^2 &= 
    \norm{\Tprm_t + (\gamma_t -\gamma_{t+1})e_t - \gamma_{t+1} \left(\grd \ell(\prm_t; z_{t+1})+ \zeta_{t+1} \right) - \prm_{PS}}^2
    \\
    &= \norm{\Tprm_t - \prm_{PS}}^2 + \norm{(\gamma_t - \gamma_{t+1}) e_t -\gamma_{t+1} [\grd \ell(\prm_t; z_{t+1}) + \zeta_{t+1}]}^2 
    \\
    &\quad + 2\Pscal{\Tprm_t-\prm_{PS}}{(\gamma_t - \gamma_{t+1}) e_t -\gamma_{t+1} [\grd \ell(\prm_t; z_{t+1}) + \zeta_{t+1}]}.
\end{align*}
Tacking conditional expectation $\EE_t[\cdot]$ on both sides leads to
\begin{align}\label{eq:ee}
    \EE_t \norm{\Tprm_{t+1} - \prm_{PS}}^2 & \leq \norm{\Tprm_t - \prm_{PS}}^2 + 2\norm{(\gamma_t - \gamma_{t+1}) e_t}^2
    \\
    & \quad + 2\gamma_{t+1}^2 \left[ \EE_t \norm{\grd \ell(\prm_t; z_{t+1})}^2 + d\sigmaDP{2} \right] + {2\Pscal{\Tprm_t - \prm_{PS}}{(\gamma_t - \gamma_{t+1}) e_t}}  \notag
    \\
    &\quad - 2\gamma_{t+1} \Pscal{\Tprm_t - \prm_{PS}}{\grd f(\prm_t; \prm_t) - \grd f(\prm_{PS}, \prm_{PS}) },
    \notag
\end{align}
where we use inequality $(a+b)^2 \leq 2a^2 + 2b^2$,  $\zeta_{t+1}\sim {\cal N}(0, \sigmaDP{2} {\bf I})$ and the fact that $\grd f(\prm_{PS}, \prm_{PS})=0.$

\blue{Applying \eqref{eq:gb_bound} leads to following upper bound: 
\begin{align}
    \EE_t \normtxt{\Tprm_{t+1} - \prm_{PS}}^2 & \leq (1 + 2 \gamma_{t+1}^2 B^2 ) \norm{\Tprm_t - \prm_{PS}}^2 + 2(\gamma_t - \gamma_{t+1})^2 \norm{ e_t}^2 
    \\
    &\quad + 2\gamma_{t+1}^2 \left[ G^2 + d\sigmaDP{2} \right] + \underbrace{2\Pscal{\Tprm_t - \prm_{PS}}{(\gamma_t - \gamma_{t+1}) e_t}}_{\eqdef B_1} \notag
    \\
    &\quad - 2\gamma_{t+1} \underbrace{\Pscal{\Tprm_t - \prm_{PS}}{\grd f(\prm_t; \prm_t) - \grd f(\prm_{PS}, \prm_{PS}) }}_{\eqdef B_2}. \notag
\end{align}
}
Now, let's consider the term $B_1$ first. Using Holder's inequality, we have
\begin{align*}
    B_1 \leq 4 (\gamma_t - \gamma_{t+1})
    \left[ \norm{\Tprm_t - \prm_{PS}}^2 + \norm{e_t}^2 \right].
\end{align*}
For the term $B_2$, we observe the following chain,
\begin{align*}
    B_2 &= \Pscal{\Tprm_t - \prm_{PS}}{\grd f(\prm_t; \prm_t) - \grd f(\Tprm_t; \prm_{PS}) + \grd f(\Tprm_t; \prm_{PS}) - \grd f(\prm_{PS}; \prm_{PS})}
    \\
    &\overset{(a)}{\geq} \mu \norm{\Tprm_t - \prm_{PS}}^2 + \Pscal{\Tprm_t - \prm_{PS}}{\grd f(\prm_t; \prm_{t}) - \grd f(\Tprm_t; \prm_{PS})}
    \\
    &= \mu \norm{\Tprm_t - \prm_{PS}}^2 - \Pscal{\prm_{PS} - \Tprm_t}{\grd f(\prm_t; \prm_t) - \grd f(\Tprm_t; \Tprm_t)} - \Pscal{ \prm_{PS} - \Tprm_t}{ \grd f(\Tprm_t; \Tprm_t)- \grd f(\Tprm_t; \prm_{PS})}
    \\
    &\overset{(b)}{\geq} \mu \norm{\Tprm_t - \prm_{PS}}^2 - \norm{\Tprm_t - \prm_{PS}} \cdot L(1+\beta) \norm{\prm_t - \Tprm_t} -  L\beta \norm{\Tprm_t - \prm_{PS}}^2
    \\
    &\overset{(c)}{=} \tmu \norm{\Tprm_t - \prm_{PS}}^2 - L(1+\beta ) \norm{\Tprm_t - \prm_{PS}} \cdot  \norm{\gamma_t e_t}
    \\
    &\geq \tmu \norm{\Tprm_t - \prm_{PS}}^2 - L(1+\beta) \cdot \left( \frac{\tmu}{2L(1+\beta)}  \norm{\Tprm_t - \prm_{PS}}^2 + \frac{2L(1+\beta)}{\tmu} \gamma_t^2 \norm{e_t}^2 \right)
    \\
    &= \frac{\tmu}{2} \norm{\Tprm_t - \prm_{PS}}^2 - \frac{2L^2 (1+\beta)^2}{\tmu} \gamma_t^2 \norm{e_t}^2 ,
\end{align*}
\noindent where in inequality $(a)$, we have used A\ref{assu:lips},  in inequality (b), we have used A\ref{assu:w1}
and Lemma D.4 of \citep{perdomo2020performative}, and in equality $(c)$, we use the definition of $\Tprm_t = \prm_t - \gamma_t e_t$.

\blue{
Substituting above upper bound for $B_1$ and $B_2$ to (\ref{eq:ee}) gives us
\begin{equation}
\begin{aligned}
    \EE_t \norm{\Tprm_{t+1} - \prm_{PS}}^2 &\leq \left( 1 - \tmu \gamma_{t+1} + \blue{2} B^2 \gamma_{t+1}^2 + 4(\gamma_t - \gamma_{t+1})\right)\norm{\Tprm_t - \prm_{PS}}^2 + 2(G^2+d\sigmaDP{2}) \gamma_{t+1}^2 
    \\
    &\quad + \left( 6(\gamma_t - \gamma_{t+1})^2 + \frac{4L^2(1+\beta)^2}{\tmu}\gamma_t^2 \gamma_{t+1} \right)\norm{e_t}^2.
\end{aligned}
\end{equation}
Setting $\gamma_{t+1} \leq \tmu / \blue{(4B^2)}$ ensures that
\begin{equation}\label{eq:gg}
\begin{aligned}
    \EE_t \norm{\Tprm_{t+1} - \prm_{PS}}^2 &\leq \left( 1 - (\tmu/2) \gamma_{t+1} + 4(\gamma_t - \gamma_{t+1})\right)\norm{\Tprm_t - \prm_{PS}}^2 + 2(G^2+d\sigmaDP{2}) \gamma_{t+1}^2 
    \\
    &\quad + \left( 6(\gamma_t - \gamma_{t+1})^2 + \frac{4L^2(1+\beta)^2}{\tmu}\gamma_t^2 \gamma_{t+1} \right)\norm{e_t}^2.
\end{aligned}
\end{equation}}

If $\gamma_{t}\equiv \gamma$ is constant step size, then $\gamma_t - \gamma_{t+1} = 0$.
For diminishing step size case $\gamma_t = \frac{a_0}{a_1 + t}$, if $a_0\geq \frac{1}{b}$, $b>0$, we can prove that $\gamma_t - \gamma_{t+1} \leq b \gamma_{t+1}^2$. Therefore, (\ref{eq:gg}) becomes
\begin{align*}
    \EE_t \norm{\Tprm_{t+1} - \prm_{PS}}^2 &\leq \left( 1 - \blue{(\tmu/2)} \gamma_{t+1} + 4 b \gamma_{t+1}^2\right)\norm{\Tprm_t - \prm_{PS}}^2 + 2(G^2+d\sigmaDP{2}) \gamma_{t+1}^2  \\
    &\quad + \left( 6 b^2 \gamma_{t+1}^4 \frac{4L^2(1+\beta)^2}{\tmu}\gamma_t^2 \gamma_{t+1} \right)\norm{e_t}^2.
\end{align*}
If $\sup_{t\geq 1} \gamma_{t} \leq \frac{\tmu}{16 b}$
then we have $1 - (\tmu/2) \gamma_{t+1} + 4 b \gamma_{t+1}^2 \leq 1 - \blue{(\tmu/4)} \gamma_{t+1} $. Thus,
\begin{align*}
    \EE_t \norm{\Tprm_{t+1} - \prm_{PS}}^2 &\leq \left( 1 - (\tmu/\blue{4}) \gamma_{t+1} \right)\norm{\Tprm_t - \prm_{PS}}^2 + 2(G^2+d\sigmaDP{2}) \gamma_{t+1}^2 
    \\
    &\quad + \left( 6 b^2 \gamma_{t+1}^4 + \frac{4L^2(1+\beta)^2}{\tmu}\gamma_t^2 \gamma_{t+1} \right)\norm{e_t}^2.
\end{align*}
Taking full expectation on both sides and applying A\ref{assu:et} leads to
\begin{align*}
    \EE \norm{\Tprm_{t+1} - \prm_{PS}}^2 &\leq \left( 1 - \blue{(\tmu/4)} \gamma_{t+1} \right)\norm{\Tprm_t - \prm_{PS}}^2 + 2(G^2+d\sigmaDP{2}) \gamma_{t+1}^2 +  6 b^2 M^2 \gamma_{t+1}^4 
    \\
    &\quad + \frac{4L^2M^2(1+\beta)^2}{\tmu}\gamma_t^2 \gamma_{t+1}. 
\end{align*}
Suppose that $\sup_{t\geq 1} \gamma_t^2/\gamma_{t+1}^2 \leq 1+ \bar{b} \gamma_{t+1}^2$, then we have
\begin{align*}
    \EE \norm{\Tprm_{t+1} - \prm_{PS}}^2 &\leq \left( 1 - \blue{(\tmu/4)} \gamma_{t+1} \right)\norm{\Tprm_t - \prm_{PS}}^2 + 2(G^2+d\sigmaDP{2}) \gamma_{t+1}^2 + \frac{4L^2M^2(1+\beta)^2}{\tmu} \gamma_{t+1}^3 
    \\
    &\qquad
    +  6 b^2 M^2 \gamma_{t+1}^4  + \frac{4L^2M^2 \bar{b}(1+\beta)^2}{\tmu}\gamma_{t+1}^5.
\end{align*}
Solving the above recursion leads to 
\begin{align}\label{eq:contraction_eq_thm_dicesgd}
    \EE \norm{\Tprm_{t+1} - \prm_{PS}}^2 &\leq \prod_{i=1}^{t+1}\left( 1 -\blue{(\tmu/4)} \gamma_{i} \right)\norm{\Tprm_0 - \prm_{PS}}^2 + 2(G^2+d\sigmaDP{2}) \sum_{i=1}^{t+1} \gamma_{i}^2 \prod_{j=i+1}^{t+1}(1-\frac{\tmu}{\blue{4}}\gamma_i)  
    \nonumber\\
    &\qquad+ \frac{4L^2M^2(1+\beta)^2}{\tmu} \sum_{i=1}^{t+1} \gamma_i^3 \prod_{j=i+1}^{t+1}(1-\frac{\tmu}{\blue{4}} \gamma_i)  + 6 b^2 M^2 \sum_{i=1}^{t+1} \gamma_i^4 \prod_{j=i+1}^{t+1}(1-\frac{\tmu}{\blue{4}} \gamma_i) \nonumber\\
    &\qquad+ \frac{4L^2M^2 \bar{b}(1+\beta)^2}{\tmu} \sum_{i=1}^{t+1} \gamma_i^5 \prod_{j=i+1}^{t+1}(1-\frac{\tmu}{\blue{4}} \gamma_i)
    \nonumber\\
    &\leq \prod_{i=1}^{t+1}\left( 1 - \blue{(\tmu/4)} \gamma_{i} \right)\norm{\Tprm_0 - \prm_{PS}}^2 + \frac{\blue{8}(G^2 + d\sigmaDP{2})}{\tmu}\gamma_{t+1} + \frac{\blue{16}L^2M^2(1+\beta)^2}{\tmu^2}\gamma_{t+1}^2 \nonumber\\
    &\qquad+ \frac{\blue{24} b^2 M^2}{\tmu} \gamma_{t+1}^3 + \frac{\blue{16}L^2M^2 \bar{b}(1+\beta)^2}{\tmu^2} \gamma_{t+1}^4,
\end{align}
where the last inequality follows from Lemma \ref{lem:aux}.
Thus, we can now conclude the proof.

\section{Proof of Theorem~\ref{thm:ncvx-dice}} \label{app:dicesgd-ncvx}

\begin{manualtheoremT}{\ref{thm:ncvx-dice} (formal)}
    Under A\ref{assu:lips}, \ref{assu:var-ncvx}, \ref{assu:w1}, \ref{assu:bd_loss}, \ref{assu:et}. Let $\gamma$ denotes the constant step size. 
    If it holds that $\gamma \leq 1 / (2 L (1+\sigma_1^2))$, then the iterates generated by {\dicesgd} admits the following bound for any $T \geq 1$,
    \begin{align}
        \min_{t=0,...,T-1} \EE\left[ \| \grd f(\prm_{t}; \prm_{t}) \|^2 \right] &\leq \frac{4 \Delta_0 }{T\gamma} + {\sf b} \beta + {2 L \gamma} \left[ \sigmaDP{2} + \sigma_0^2 \right] +  2 L^2 M^2 \gamma^2,
    \end{align}
    where  $\Delta_0 \eqdef \EE \left[ f(\prm_0 - \gamma e_0; \prm_0) - \ell^\star \right]$ is an upper bound for the initial error and the constant ${\sf b}$ is defined as 
    \begin{align*}
        {\sf b} &\eqdef 4 \ell_{\max} \left( C_1 +C_2 + \sqrt{d} \sigmaDP{} \right).
    \end{align*}
\end{manualtheoremT}
In particular, for sufficiently large $T$, setting $\gamma = 1 / \sqrt{T}$ yields the bound in \eqref{eq:dicesgd_noncvx} of the main paper.

\begin{proof}
\footnote{We notice that our proof differs from that of \citep[Theorem 3.6]{zhang2023differentially} and is simpler than the latter due to the addition of A\ref{assu:et}.}Consider the {\dicesgd} algorithm with constant step size $\gamma_{t+1} = \gamma$. Similar to \S\ref{app:dice}, we define $\Tprm_t = \prm_t - \gamma e_t$ and notice the following recursion:
\[
\Tprm_{t+1} = \Tprm_t - \gamma \left( \zeta_{t+1} + \grd \ell(\prm_t ; Z_{t+1} ) \right)
\]
To facilitate the derivations, we define the shorthand notation: $\grd f_{t} \eqdef \grd f(\prm_{t}; \prm_{t})$. 
For any fixed sample $z\in {\sf Z}$, applying A\ref{assu:lips} leads to the following upper bound:
\begin{align*}
    \ell(\Tprm_{t+1}, z) &\leq \ell(\Tprm_t, z) + \pscal{\grd \ell(\Tprm_t, z)}{\Tprm_{t+1} - \Tprm_{t}} + \frac{L}{2} \norm{\Tprm_{t+1} - \Tprm_{t}}^2 
\end{align*}
Taking integration on the fixed $z$ with weights given by the p.d.f. of ${\cal D}(\prm_{t})$ on the both sides of above inequality yields
\begin{align*}
    f( \Tprm_{t+1}, \prm_t ) &\leq f(\Tprm_t, \prm_t) + \pscal{\grd f (\Tprm_t, \prm_t) }{\Tprm_{t+1} - \Tprm_{t}} + \frac{L}{2} \norm{\Tprm_{t+1} - \Tprm_{t}}^2 \\
    & = f(\Tprm_t, \prm_t) - \gamma \pscal{\grd f (\Tprm_t, \prm_t) }{ \zeta_{t+1} +\grd \ell( \prm_t; Z_{t+1} ) } + \frac{L \gamma^2}{2} \norm{ \zeta_{t+1} +\grd \ell( \prm_t; Z_{t+1} ) }^2
\end{align*}
Taking the conditional expectation $\EE_t [\cdot]$ yields
\begin{align} \label{eq:dicesgd_recur}
    \EE_t [ f( \Tprm_{t+1}, \prm_t ) ] & \leq f(\Tprm_t, \prm_t) - \gamma \pscal{\grd f (\Tprm_t, \prm_t) }{ \grd f( \prm_t; \prm_t ) } + \frac{L \gamma^2}{2} \EE_t \left[ \norm{ \zeta_{t+1} +\grd \ell( \prm_t; Z_{t+1} ) }^2 \right]
\end{align}
Notice that 
\begin{align*}
\EE_t \left[ \norm{ \zeta_{t+1} +\grd \ell( \prm_t; Z_{t+1} ) }^2 \right] & \leq \sigmaDP{2} + \sigma_0^2 + (1 +\sigma_1^2) \| \grd f( \prm_t; \prm_t ) \|^2
\end{align*}
and under A\ref{assu:et}, it holds that 
\begin{align*}
    - \EE \pscal{\grd f (\Tprm_t, \prm_t) }{ \grd f( \prm_t; \prm_t ) }
    & = - \EE \| \grd f_t \|^2 + \EE \pscal{\grd f (\Tprm_t, \prm_t) - \grd f(\prm_t;\prm_t) }{ \grd f(\prm_t;\prm_t) } \\
    & \leq - \frac{1}{2} \EE \| \grd f_t \|^2 + \frac{1}{2} \EE \| \grd f (\Tprm_t, \prm_t) - \grd f(\prm_t;\prm_t) \|^2 \\
    & \overset{(a)}{\leq} - \frac{1}{2} \EE \| \grd f_t \|^2 + \frac{L^2}{2} \EE \| \Tprm_t - \prm_t \|^2 \\
    & \leq - \frac{1}{2} \EE \| \grd f_t \|^2 +\frac{L^2}{2} \EE \| \gamma e_t \|^2 \leq - \frac{1}{2} \EE \| \grd f_t \|^2 +\frac{L^2 M^2 \gamma^2 }{2} 
\end{align*}
where $(a)$ is due to Lemma \ref{lem:continuity}.
Substituting back into \eqref{eq:dicesgd_recur} and taking full expectation lead to 
\begin{align*} 
    \EE [ f( \Tprm_{t+1}, \prm_t ) ] & \leq \EE[ f(\Tprm_t, \prm_t) ] -  \frac{\gamma}{2} \left( 1 - L \gamma (1 +\sigma_1^2) \right) \EE[ \| \grd f_t \|^2 ] + \frac{L \gamma^2}{2} \left[ \sigmaDP{2} + \sigma_0^2 \right] + \frac{ L^2 M^2 \gamma^3 }{ 2 } \\
    & \leq \EE[ f(\Tprm_t, \prm_t) ] -  \frac{\gamma}{4} \EE[ \| \grd f_t \|^2 ] + \frac{L \gamma^2}{2} \left[ \sigmaDP{2} + \sigma_0^2 \right] + \frac{ L^2 M^2 \gamma^3 }{ 2 }
\end{align*}
where the last inequality is due to $\gamma \leq 1 / (2 L (1+\sigma_1^2))$.
Recall that we have bounded a similar term in \eqref{eq:ft-ft1}:
\begin{align*}
     \EE\left[ f(\Tprm_{t+1}; \prm_t) - f(\Tprm_{t+1}; \prm_{t+1}) \right] \leq \ell_{\max} \beta \norm{\prm_{t} - \prm_{t+1}} \leq  \gamma \ell_{\max} \beta \left( (C_1 +C_2) + \sqrt{d} \sigmaDP{} \right),
\end{align*}
under A\ref{assu:tv} \& \ref{assu:bd_loss}.
This yields 
\begin{align*}
    \frac{\gamma}{4} \EE[ \|\grd f_t\|^2 ] \leq \EE[ f(\Tprm_t, \prm_t) - f( \Tprm_{t+1}, \prm_{t+1} ) ] + \gamma \ell_{\max} \beta \left( C_1 +C_2 + \sqrt{d} \sigmaDP{} \right) + \frac{L \gamma^2}{2} \left[ \sigmaDP{2} + \sigma_0^2 \right] + \frac{ L^2 M^2 \gamma^3 }{ 2 }
\end{align*}
Summing up the inequality from $t=0$ to $t=T-1$ and dividing by $T \gamma / 4$ yields
\begin{align*}
    \frac{1}{T} \sum_{t=0}^{T-1} \EE[ \|\grd f_t\|^2 ] \leq \frac{4}{T\gamma} \left[ f(\Tprm_0, \prm_0) - \ell^\star \right] + 4 \ell_{\max} \beta \left( C_1 +C_2 + \sqrt{d} \sigmaDP{} \right) + {2 L \gamma} \left[ \sigmaDP{2} + \sigma_0^2 \right] +  2 L^2 M^2 \gamma^2 
\end{align*}
For sufficiently large $T$, setting $\gamma = 1 / \sqrt{T}$ yields the bound in the desired theorem.
\end{proof}

\section{Details of Numerical Experiments}\label{app:addexp}

This section provides additional details for the numerical experiments in \S \ref{sec:experiment}. 

\subsection{Validating A\ref{assu:et} Empirically}
\begin{figure}[htbp]
    \centering\includegraphics[width=.34\textwidth]{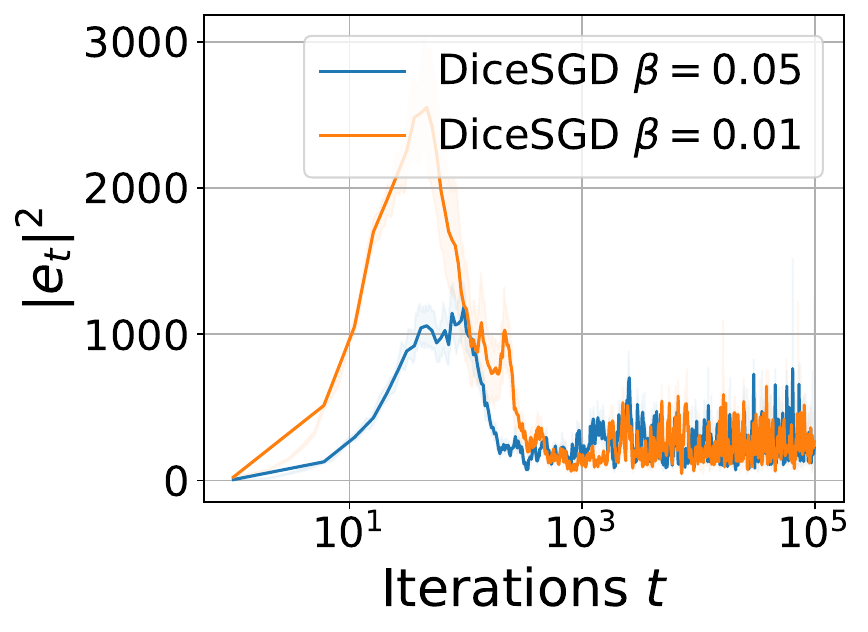}
    \includegraphics[width=.33\textwidth]{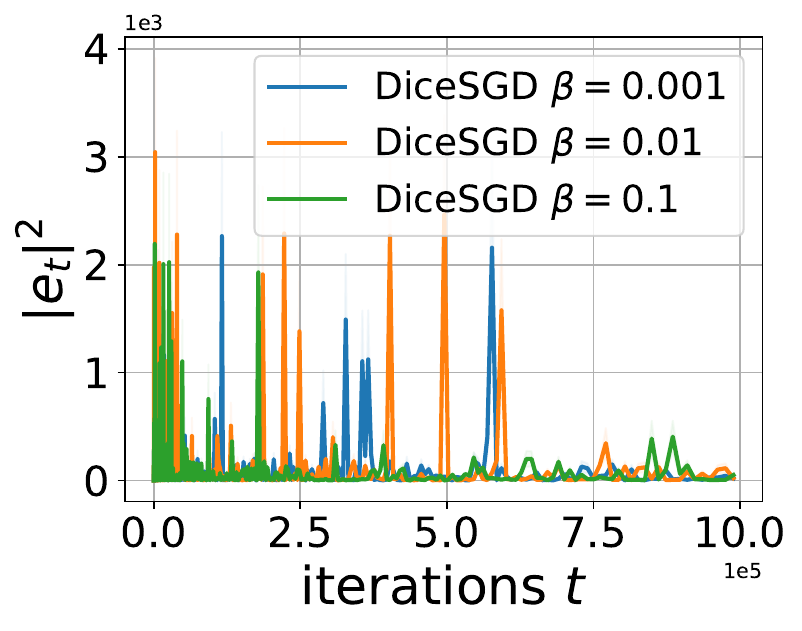}
\vspace{-.2cm}
\caption{Behavior of $e_{t}$ with {\dicesgd} for quadratic minimization (\emph{left}) and logistic regression (\emph{right}).}
    \label{fig:et}
\end{figure}
We plot the average of $ \norm{e_t}^2 $ for {\dicesgd} as a function of $t$ in Fig.~\ref{fig:et}, and observe that it is always bounded for both of our numerical examples, thus verifying our A\ref{assu:et}.

\subsection{Additional Details of Numerical Experiments}

We provide additional details for the numerical experiments on the logistics regression problem. 

In Fig.~\ref{fig:appendix-simu} (Left), we plot the performative risk $\EE_{z\sim {\cal D}(\prm_t)}[\ell(\prm_t; z)] $ as a function of the number of iterations $t$. We observe that {\dicesgd} can achieve lower performative risk at a faster rate compared to {\pcsgd} algorithm. Furthermore, increasing sensitivity $\beta \uparrow$ leads to lower train true positive rate, which is compatible with our Theorems \ref{thm1} and \ref{thm:dicesgd}.

In the training set, records with positive labels $(y=1)$ account for 6.624\%, while negative samples $(y=0)$ comprise 93.37\%. To evaluate model performance, we have presented the test true negative/positive rate in Fig.~\ref{fig2} (second \& third). The training true negative/positive rate is shown in Fig.~\ref{fig:appendix-simu} \emph{(Middle\& Right)}. For self-completeness, we recall the definition positive (negative) label accuracy as follows: 
\[
\text{True Positive (Negative) rate} = \frac{\mathrm{TP}_{\text{pos/neg}}}{\mathrm{TP}_{\text{pos/neg}} + \mathrm{FN}_{\text{pos/neg}}}
\]
where $\mathrm{TP}_{\text{pos/neg}}$ denotes the number of samples with positive (negative) labels correctly classified as positive (negative), whereas $\mathrm{FN}_{\text{pos/neg}}$ represents the number of samples with positive (negative) labels incorrectly classified as negative (positive). 

\begin{figure*}[htbp]
    \includegraphics[width=0.33\linewidth]{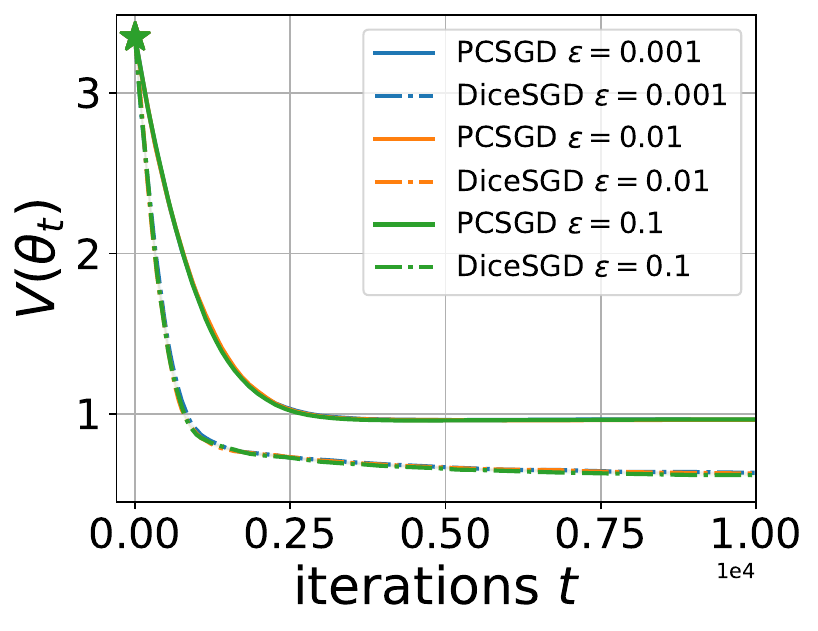}
    \includegraphics[width=.33\textwidth]{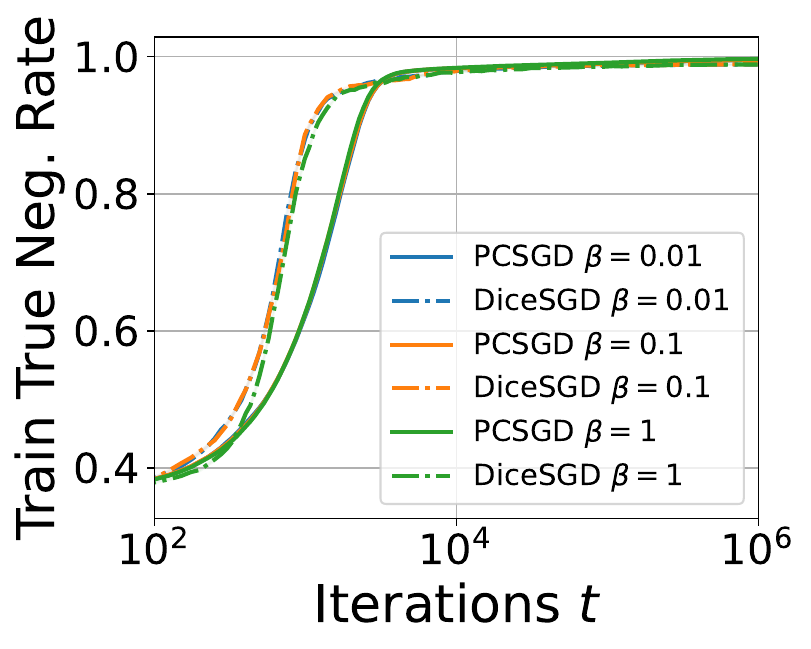}
    \includegraphics[width=.33\textwidth]{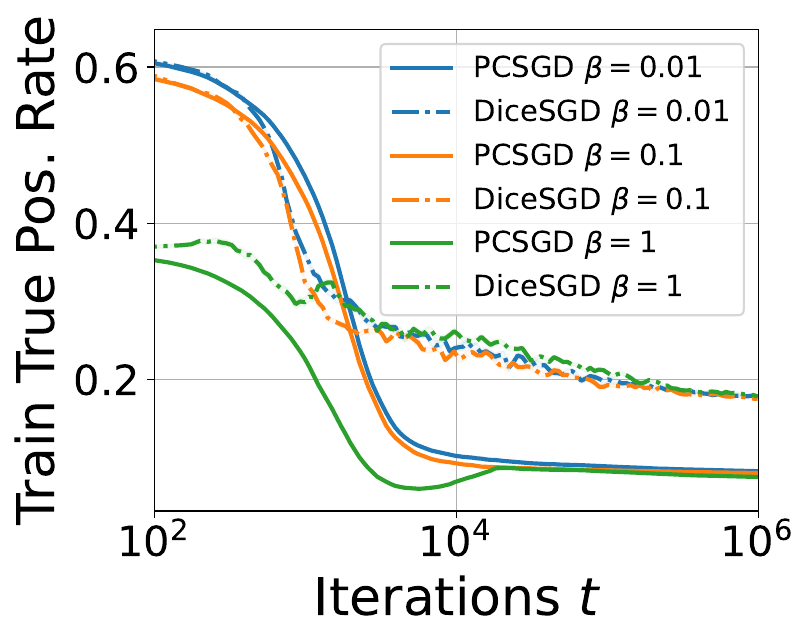}
    \caption{{\bf Logistic Regression} \emph{(Left)}: Performative Risk $V(\prm)$. \emph{(Middle) \& (Right)}: Train true neg./pos. rate. }
    \label{fig:appendix-simu}
\end{figure*}



\end{document}